\documentclass[12pt, reqno]{amsart}
\usepackage{mathrsfs}
\usepackage{amsfonts}
\usepackage{}
\usepackage{amsmath}
\usepackage{amsthm}
\usepackage{amssymb}
\usepackage[compress,numbers]{natbib}
\usepackage{graphicx}
\usepackage{comment}
\usepackage{tikz}
\newtheorem{theorem}{Theorem}[section]
\newtheorem{lemma}[theorem]{Lemma}

\theoremstyle{definition}
\newtheorem{definition}[theorem]{Definition}

\newtheorem{remark}[theorem]{Remark}
\newtheorem{problem}[theorem]{Problem}
\numberwithin{equation}{section}

\begin{document}
	
	\title{The Gaussian  Minkowski type  problem for $C$-pseudo-cones }
	
	\author{Junjie Shan, Wenchuan Hu, Wenxue Xu$^{*}$ }
	
	\address{School of Mathematics, Sichuan University, Chengdu, Sichuan, 610000, PR China}

	\email{shanjjmath@163.com,   wenchuan@scu.edu.cn
	}
	
	\address{School of Mathematics and Statistics,
		Southwest University,
		Chongqing 400715, PR China}

	\email{xuwenxue83@swu.edu.cn
	}

	\thanks{{\it 2020 Math Subject Classifications}:  52A40, 52A38}
	

	\thanks{This paper was supported by  the  Natural Science Foundation of Chongqing
		(No. cstc2021jcyj-msxmX1057).
	}

	\thanks{ $^{*}$ Corresponding author}
\begin{abstract}
The Gaussian surface area measure and the Gaussian cone measure  for $C$-pseudo-cones are introduced and their corresponding Gaussian Minkowski  problem and Gaussian log-Minkowski  problem are proposed, respectively.
The existence and  uniqueness of solutions to these problems for $C$-pseudo-cones  are established.

\end{abstract}

\maketitle

\section{Introduction}
The classical Minkowski problem, first studied by Minkowski for polytopes in 1897,   aims to answer the question: What are the necessary and sufficient conditions for a Borel measure on the unit sphere to be the surface area measure of a \emph{convex body} (i.e., compact convex subsets of $\mathbb R^n$ with nonempty interiors)?
In recent decades, the theory of
 Minkowski problems,  connecting geometric invariants to corresponding geometric measures,  has expanded considerably, some significant advances have been made in the framework of the
Brunn-Minkowski theory and its extensions for convex bodies, for example,  the $L_p$ Minkowski problem \cite{Lutwak1993}, the logarithmic Minkowski problem \cite{LYZ2013JAMS},  the dual Minkowski problem \cite{HLYZdual}, the Gauss Minkowski problem \cite{Huang2021}, the chord Minkowski problem \cite{ChordCPAM} and so on.
 The new paper \cite{ZhangBAMS} surveys classical and new results on Minkowski problems and their related subjects.


Unlike  the Minkowski problems associated with convex bodies mentioned above,  another type of  Minkowski problem related to unbounded closed convex sets in $\mathbb R^n$ has   also recently been studied  by Li-Ye-Zhu \cite{Lyzdualcone}, Schneider \cite{schneider2021,schneider pescone,schneider weighted cone,schneider weighted},  Zhang \cite{zhang asyp}, Ai-Yang-Ye \cite{ai lpdual} following from the pioneering works of establishing the Brunn-Minkowski theory for $C$-coconvex sets by  Milman-Rotem \cite{MilmanRotem}, Schneider \cite{schneider 2018}, and Yang-Ye-Zhu \cite{yangyezhuLp2022}. In this paper, unbounded closed convex sets of interest are $C$-pseudo-cones
introduced  by Schneider \cite{schneider pescone, schneider weighted cone}. The set of pseudo-cones can be seen as a counterpart to the set of convex bodies containing the origin in the interior. A\emph{ pseudo-cone} $K$ is a  nonempty closed convex set  that does not contain the origin and satisfies $\lambda K\subseteq K$ for $\lambda \ge1$. Its \emph{recession cone} is   given by
$$
\operatorname{rec} K=\left\{z \in \mathbb{R}^{n}: K+z \subseteq K\right\}.$$
A \textit{$C$-pseudo-cone} is a pseudo-cone with recession cone $C$.   Intuitively, its asymptotic behavior is controlled by $C$. Then, it is easy to see that $C$-pseudo-cones are a special class of unbounded convex sets.

 For a cone $C$, we always assume that it is pointed (not  containing any line) and  $n$-dimensional. Let $C^{\circ}$  be the \textit{dual cone} of a convex cone $C$ defined by
$ C^{\circ}=\left\{x \in \mathbb{R}^{n}:\langle x, y\rangle \leq 0 \text { for all } y \in C\right\},$
where $\langle x, y\rangle$ denotes the inner product of $x,y\in\mathbb R^n$.
For a given $C$, the subsets
\begin{align*}
 \Omega_{C}=S^{n-1}\cap \operatorname{int} C, \quad\Omega_{C^{\circ}}=S^{n-1}\cap \operatorname{int} C^{\circ}
 \end{align*}
of the unit sphere $ S^{n-1}$ (where $\text{int} \ C$ is the interior of $C$) will play an important role.

For $C$-pseudo-cones, the Minkowski problem asks: \emph{What are the necessary and sufficient conditions for a Borel measure $\mu$ on $\Omega_{C^{\circ}}$ to be the surface area measure of a $C$-pseudo-cone?} However, the solutions to this problem have not yet been completely solved. Unlike the surface area measure of convex bodies, that of a $C$-pseudo-cone may become infinite, presenting substantial challenges in resolving the Minkowski problem. Nevertheless, significant progress has been achieved for finite Borel measures $\mu$ on $\Omega_{C^{\circ}}$. Notably, Schneider \cite{schneider2021} demonstrated that every nonzero finite $\mu$ is the surface area measure of a $C$-pseudo-cone $K$, where the complement $C\backslash K$ has finite volume. The $L_{p}$-Minkowski type theorem for pseudo-cones was further studied by Yang-Ye-Zhu \cite{yangyezhuLp2022}, with versions in the dual setting  developed by Li-Ye-Zhu \cite{Lyzdualcone}.
The $(p,q)$-dual curvature measures were later introduced by Ai-Yang-Ye \cite{ai lpdual}.

However, for infinite measures, necessary and sufficient conditions are still unknown. As Schneider pointed out in \cite{schneider weighted cone, schneider weighted}, the shape of a  $C$-pseudo-cone is strongly influenced by the shape of $C$ if the distance from the origin tends to infinity. In order to neglect those regions where the effect on the shape of a $C$-pseudo-cone is relatively weak, yielding a finite measure, it seems that suitable weightings are even more desirable than in the context of convex bodies. By this insight, the \emph{homogeneous-weighted surface area measure} $S_{n-1}^\Theta(K,\cdot)$ of  a $C$-pseudo-cone $K$ was considered by Schneider \cite{schneider weighted cone}, which is defined by
\begin{align}\label{weight-Schneider}
S_{n-1}^\Theta(K,\omega)=\int_{\nu_{K}^{-1}(\omega)}\Theta(x) d \mathcal{H}^{n-1}(x)
\end{align}
for Borel set $\omega\subset\Omega_{C^{\circ}}$, where $\Theta:C \setminus \{o\}\to(0,\infty)$ is continuous and homogeneous of degree $-q$ with $n-1<q<n$,  $\nu_{K}$ is the outer unit normal vector  of $K$ at points $x\in \partial K$. Then the measure $S_{n-1}^\Theta(K,\cdot)$ is finite and the weighted Minkowski theorem for $C$-pseudo-cones is obtained. For further references involving pseudo-cones, we refer to \cite{sz2024asyp,    xulileng 2023}.

Inspired by the weighted Minkowski theorem for $C$-pseudo-cones \cite{schneider weighted cone,schneider weighted}, a natural question arises:
\emph{What other weightings, if any, could generate a finite measure on $\Omega_{C^{\circ}}$?} Notably, the Gaussian density emerges as a natural candidate in this non-homogeneous setting. Indeed, by substituting the weighting  $\Theta(x)$ in (\ref{weight-Schneider}) with the Gaussian density $e^{-\frac{|x|^{2}}{2}}$,  we define the \textit{Gaussian surface area measure} of a $C$-pseudo-cone $K$ by
\begin{equation}
    S_{\gamma^{n}}(K,\eta)=\frac{1}{(\sqrt{2 \pi})^{n}} \int_{\nu_{K}^{-1}(\eta)} e^{-\frac{|x|^{2}}{2}} d \mathcal{H}^{n-1}(x)
\end{equation}
for Borel set $\eta\subset\Omega_{C^{\circ}}$, where $\gamma^{n}$ denotes the standard Gaussian probability measure on  Euclidean space $\mathbb{R}^{n}$. The measure $S_{\gamma^{n}}(K,\cdot)$ is clearly finite. This naturally leads to the formulation of the corresponding Minkowski problem,  called \emph{the Gaussian Minkowski problem for $C$-pseudo-cones}, which we state as follows:

\begin{problem}\label{pro}Given a nonzero finite Borel measure $\mu$ on
$\Omega_{C^{\circ}}$, is there a $C$-pseudo-cone $K$ such that $\mu=S_{\gamma^{n}}(K,\cdot)$?
\end{problem}

The Gaussian Minkowski problem in the context of convex bodies was first investigated by Huang-Xi-Zhao  \cite{Huang2021}, with its $L_p$ extension subsequently proposed by  Liu  \cite{Liu2022}.   Feng-Liu-Xu \cite{FLX2023}  later established results for the non-symmetric case. For recent advances in this field,  we refer  to \cite{zhao2023, FHX2023,hulog2024,Liv2019}.

However, the assertion in Problem \ref{pro}  fails to hold for arbitrary cones.  Indeed, the variational approaches employed in both the homogeneous weighted Minkowski problem \cite{ schneider weighted cone,schneider weighted} and the classical Gaussian Minkowski problem for convex bodies \cite{Huang2021} prove inadequate for resolving Problem \ref{pro}. To overcome this limitation, we introduce a new variational functional, through which we establish the existence of  solutions to the (normalized) Gaussian Minkowski problem for $C$-pseudo-cones as follows.

\begin{theorem}
	Let $\mu$ be a nonzero, finite Borel measure on $\Omega_{C^{\circ}}$. Then there exists a $C$-pseudo-cone $K$ with
	$$\mu=cS_{\gamma^{n}}(K,\cdot),$$
	where $c=\frac{\int_{\Omega}\bar{h}_{K}d\mu}{\gamma^{n}(K)}.$
\end{theorem}

Since  the Gaussian surface area measure does not have any homogeneity, the normalizing factor $c$ in the equation above can not be removed.
The  constant factors in these solutions are commonly recognized as the standard normalization for Minkowski problems involving non-homogeneous geometric measures and probability measures, such as the Orlicz-Minkowski type problems \cite{OrliczMP, ZhangBAMS}.

We now turn to the uniqueness part of the Gaussian Minkowski problem for $C$-pseudo-cones. As we  will show in Section 5,
similar to the case of the Gaussian Minkowski problem for convex bodies, solutions to this problem are also not unique without additional conditions. However, when we restrict our consideration to $C$-pseudo-cones sharing identical Gaussian volumes, the uniqueness part of the Gaussian Minkowski problem for $C$-pseudo-cones can be established.

\begin{theorem}\label{uniq}
	Suppose $K,L\in\mathcal{K}(C,\omega)$ for some nonempty compact $\omega\subset\Omega_{C^{\circ}}$, $\gamma^{n}(K)=\gamma^{n}(L)$, if
	$$S_{\gamma^{n}}(K,\cdot)=S_{\gamma^{n}}(L,\cdot),$$
	 then $K=L$.
\end{theorem}

Here $\mathcal{K}(C,\omega)$ represents a specialized class of $C$-pseudo-cones  whose geometric configurations are intrinsically determined  by the compact set  $\omega\subset\Omega_{C^{\circ}}$ (see Section \ref{section2} for details). It is worth pointing out that
the volume restriction in Theorem \ref{uniq} can be relaxed to $\gamma^{n}(K),\gamma^{n}(L)\le \frac{1}{2}\gamma^{n}(C) $ in discrete cases,
while still ensuring the uniqueness result.

Another fundamental type of Minkowski problem---the logarithmic Minkowski problem,  which characterizes the cone-volume measure, arises as  the limiting case $p=0$ of the general  $L_p$ Minkowski problem. For a convex body $K$ in $\mathbb{R}^{n}$, the cone-volume measure $V_{K}$ is a Borel measure on the unit sphere $S^{n-1}$ defined by
\begin{equation*}
	 V_{K}(\omega)=\frac{1}{n}\int_{\nu_{K}^{-1}(\omega)}\langle x,\nu_{K}(x)\rangle d \mathcal{H}^{n-1}(x)
\end{equation*}
for every Borel set $\omega \subset S^{n-1}$. The term ``cone-volume measure'' originates from the fact that $V_{K}(\omega)$ is the volume of the union of  segments connecting the origin to $\nu_{K}^{-1}(\omega)$.    For more references involving the cone-volume measure for convex bodies, see, e.g., \cite{BH logcenterbody,KH logstable,BHZ log2016,CLZ log2019,Stancu log 2002,Zhu log2014}.

We consider the Gaussian cone measure for $C$-pseudo-cones. Let $K$ be a $C$-pseudo-cone in $\mathbb{R}^{n}$, the \textit{Gaussian cone measure} $C_{\gamma^{n}}(K,\cdot)$ is  a Borel measure on $\Omega_{C^{\circ}}$  defined by
\begin{equation*}
	C_{\gamma^{n}}(K,\eta):=\frac{1}{(\sqrt{2 \pi})^{n}} \int_{\nu_{K}^{-1}(\eta)}|\langle x,\nu_{K}(x) \rangle| e^{-\frac{|x|^{2}}{2}} d \mathcal{H}^{n-1}(x)=-\int_{\eta}h_{K}(u)d	 S_{\gamma^{n}}(K,u)
\end{equation*}
for every Borel set $\eta\subset\Omega_{C^{\circ}}$, where $h_{K}$ denotes the support function of $K$. Geometrically,  the Gaussian cone measure  quantifies  the product of the height  (via the support function) and the Gaussian surface area. However,  in contrast to  the classical cone-volume measure,  $C_{\gamma^{n}}(K,\eta)$ does not represent the true Gaussian  volume  of the set of  segments connecting the origin and $\nu_{K}^{-1}(\omega)$,  due to the lack of   homogeneity. This distinction motivates our use of the term cone measure rather than cone-volume measure.

The Gaussian cone measure of $C$-pseudo-cones can actually be generated by the variation of Gaussian volume   with respect to a logarithmic
family of Wulff shapes, and thus the corresponding Minkowski problem  is termed  the \textit{Gaussian log-Minkowski problem}.  Notably, due to the restrictive geometric structure imposed by $C$ on $C$-pseudo-cones, the log-Minkowski problem associated with $C$-pseudo-cones does not require the subspace concentration condition, a critical constraint in the framework of convex bodies.
\begin{theorem}
	Let $\mu$ be a nonzero, finite Borel measure on $\Omega_{C^{\circ}}$. Then there exists a $C$-pseudo-cone $K$ such that
	$$\frac{C_{\gamma^{n}}(K,\cdot)}{\gamma^{n}(K)}=\mu.$$
\end{theorem}

The following theorem shows that the solutions to   both the Gaussian Minkowski problem and Gaussian log-Minkowski problem  for $C$-pseudo-cones are not unique.
\begin{theorem}
	Let $C\subset \mathbb{R}^{n}$ be a pointed, $n$-dimensional closed convex cone, and let $\omega\subset\Omega_{C^{\circ}}$ be a  nonempty compact  set.	Then  the following hold:
	\begin{enumerate}
		\item There exist $K_1, L_1 \in \mathcal{K}(C,\omega)$, such that
		\[
		S_{\gamma^n}(K_1,\cdot) = S_{\gamma^n}(L_1,\cdot) \quad \text{but} \quad K_1 \neq L_1.
		\]
		
		\item There exist  $K_2, L_2 \in \mathcal{K}(C,\omega)$, such that
		\[
		C_{\gamma^n}(K_2,\cdot) = C_{\gamma^n}(L_2,\cdot) \quad \text{but} \quad K_2 \neq L_2.
		\]
	\end{enumerate}
\end{theorem}

\section{Preliminaries}\label{section2}

Let  $\mathbb{R}^{n}$ denote $n$-dimensional Euclidean space, equipped with the standard inner product $\langle \cdot,\cdot\rangle$.  The open unit ball is denoted by $B$,
the unit sphere by $ S^{n-1}$,  and the origin by $o$.  For a general reference on the
theory of convex geometry,  we refer to the classical texts
\cite{Gardnerbook2006} and  \cite{schneiderbook2014}.

A set $C\subset \mathbb{R}^{n}$ is a \textit{cone} if for all $x\in C$  and $\lambda \ge 0$, we have $\lambda x\in C$.  Given an $n$-dimensional pointed closed convex cone     $C \subset \mathbb{R}^{n}$, its \emph{dual cone} is defined as
$$
C^{\circ}=\left\{x \in \mathbb{R}^{n}:\langle x, y\rangle \leq 0 \text { for all } y \in C\right\}.$$
A nonempty closed convex set $K$ is called a \textit{pseudo-cone} if for all $x\in K$ and any $\lambda \ge1$, we have $\lambda x\in K$, $\lambda \ge1$. The \emph{recession cone} of a pseudo-cone $K$ is given by
$$
\operatorname{rec} K=\left\{z \in \mathbb{R}^{n}: K+z \subseteq K\right\},$$
and a \textit{$C$-pseudo-cone} is a pseudo-cone satisfying $\operatorname{rec}K=C$. We denote  $\partial_{i} K:=\partial K \cap \operatorname{int} C$.
The hyperplane and  halfspace in  $\mathbb{R}^{n}$  are defined, independently, by
$$
H(u, t)=\left\{x \in \mathbb{R}^{n}:\langle x, u\rangle=t\right\}, \quad H^{-}(u, t)=\left\{x \in \mathbb{R}^{n}:\langle x, u\rangle \leq t\right\}$$
for  $u \in S^{n-1}$  and  $t \in \mathbb{R}$.  A unit vector $\mathfrak{v}\in \operatorname{int} C \cap \operatorname{int}\left(-C^{\circ}\right)$ is fixed. For  $t>0$  we define
$$
	C^{-}(t):=C \cap H^{-}(\mathfrak{v}, t),\quad
	C^{+}(t):=C \cap H^{+}(\mathfrak{v}, t),\quad
\text{and}\quad	C(t):=C \cap H(\mathfrak{v}, t).
$$
Meanwhile, we will write $K(t)=K\cap  C(t)$, $K^{-}(t)=K\cap C^{-}(t)$ and $K^{+}(t)=K\cap C^{+}(t)$.

For any   $C$-pseudo-cone $K$, its \textit{support function} $h_{K}: C^{\circ}\rightarrow \mathbb{R}$  is defined as
\begin{equation}
	h_{K}(x):=\sup \{\langle x, y\rangle: y \in K\}.
\end{equation}
This function   is bounded and non-positive.  To obtain a positive function,  we define  the absolute support function as
$$\bar{h}_{K}:=-h_{K}.$$

The \textit{radial function}  $\varrho_{K}: \Omega_{C}\rightarrow
\mathbb{R} $ is given by
\begin{equation}
	\varrho_{K}(v):=\min \{\lambda \in \mathbb{R}: \lambda v \in K\}.
\end{equation}
This induces the \textit{radial map }: $r_{K}=\varrho_{K}(v) v$
and the \textit{radial Gauss map}: $\alpha_{K}=\nu_{K} \circ r_{K}$, where $\nu_{K}$ denotes the outer unit normal vector field on $\partial K$.

	The convergence of $C$-pseudo-cones is introduced by Schneider \cite{schneider pescone}  as follows.
	\begin{definition}\label{conv def}
		We say a sequence of  $C$-pseudo-cones  $K_{i}$  converges to a  $C$-pseudo-cone  $K$  if there exists  $t_{0}>0$  such that  $K_{i}^{-}\left(t_{0}\right) \neq \emptyset$  for all  $i$  and  $K_{i}^{-}(t) \rightarrow K^{-}(t)$  for all $t \geq t_{0}$ with respect to the Hausdorff metric.
	\end{definition}
	
	The following selection theorem for $C$-pseudo cones, established in \cite{schneider pescone}, plays a crucial role in this paper.
	\begin{lemma}\label{select}
		Let  $K_{i}$  be a sequence of  $C$-pseudo-cones. If there exist two uniformly positive constants $a$ and $b$ such that $0<a<\operatorname{dist}\left(o, K_{i}\right)<b$ for all $i$, then there exists a subsequence  $K_{i_{j}}$ converges to  a  $C$-pseudo-cone  $K$, i.e., $K_{i_{j}} \rightarrow K$.
	\end{lemma}

We denote  $\Omega:=\Omega_{C^{\circ}}=S^{n-1}\cap \operatorname{int} C^{\circ}$ in this paper.
 For $u\in\Omega$, the support hyperplane of a $C$-pseudo-cone $K$ is given by
 $$
 H_{K}(u)=\left\{x \in \mathbb{R}^{n}:\langle x, u\rangle=h_{K}(u)\right\}, \quad H^{-}_{K}(u)=\left\{x \in \mathbb{R}^{n}:\langle x, u\rangle \leq h_{K}(u)\right\}.$$

We now introduce the concept  of $C$-determined pseudo-cones and their associated properties. A $C$-pseudo-cone $K$  is  \textit{$C$-determined} by a nonempty compact set   $\omega \subset \Omega$  if
$$K=C \cap \bigcap_{u \in \omega} H_{K}^{-}(u).$$
By  $\mathcal{K}(C, \omega)$  we denote the set of all $C$-pseudo-cones that are  $C$-determined by  $\omega$. For any $K$ $\in \mathcal{K}(C, \omega)$,  the absolute support function  $\bar{h}_{K}=-h_{K}$ is positive and uniformly bounded away from zero on  $\omega$. Moreover, if  $K$  is  $C$-determined by  $\omega$, then the complement $C \backslash K$  is bounded.

 The following lemma established in \cite{schneider weighted} will be needed.
\begin{lemma}\label{compace conv}
	If  $K_{j}$ are $C$-pseudo-cones, $j \in \mathbb{N}$, satisfy  $K_{j} \rightarrow K$.   Let $ \omega \subset \Omega$  be a nonempty compact set, then
	$$K_{j}^{(\omega)} \rightarrow K^{(\omega)},$$
	where $K^{(\omega)}:=C \cap \bigcap_{u \in \omega} H_{K}^{-}(u)$ denotes the restriction of $K$ on $\omega$.
\end{lemma}
A key tool in our study is the \textit{Wulff shapes} in cones, which is constructed as follows. Given a nonempty compact set
$ \omega \subset \Omega$  and  a positive continuous function $h: \omega \rightarrow \mathbb{R}$, the Wulff shape $[h]$ is defined by
\begin{equation}\label{wulff shape}
	[h]:=C \cap \bigcap_{u \in \omega}\left\{y \in \mathbb{R}^{n}:\langle y, u\rangle \leq-h(u)\right\}.
\end{equation}
The set $[h]$ will be called the \textit{Wulff shape} associated with $(C,\omega,h)$ and  $[h]\in\mathcal{K}(C,\omega)$.

	Let $ \gamma^{n} $ represent the standard Gaussian probability measure
on Euclidean space $(\mathbb{R}^{n},|\cdot|)$
$$\gamma^{n}(E):=\frac{1}{(2 \pi)^{\frac{n}{2}}}\int_{E} e^{-\frac{|x|^{2}}{2}} d x,$$
with a measurable set $ E $.  Unlike the Lebesgue measure, the Gaussian probability measure lacks both translation invariance and homogeneity.

In the context of Gaussian measures, the isoperimetric inequality states that among all subsets of $\mathbb{R}^{n}$ with a prescribed Gaussian measure, half-spaces minimize the Gaussian perimeter  (see  \cite{Latalaicm}). For every $o$-symmetric convex bodies $ K , L $
in $ \mathbb{R}^{n} $ and any $ \lambda \in (0,1) $, the Brunn-Minkowski type inequality for the Gaussian measure was presented in \cite{Z2021} as follows
\begin{equation*}
	\gamma^{n}(\lambda K+(1-\lambda)L)^{\frac{1}{n}} \ge \lambda \gamma^{n}( K)^{\frac{1}{n}}+(1-\lambda)\gamma^{n}( L)^{\frac{1}{n}}
\end{equation*}
with equality if and only if $ K=L $. For non-symmetric convex sets,
a counterexample was provided in \cite{Nayercounterexample2013}.
Further references on inequalities
in Gaussian space can be found in
\cite{cianchideficit,Gardner2010,milman2005,KLweightedprojection,
Lehec2009,Lutwak1993,Lutwak1996}.
\\

\section{Gaussian surface area measure }
Recall that  $\Omega:=\Omega_{C^{\circ}}=S^{n-1}\cap \operatorname{int} C^{\circ}$.  Let  $K$  be a  $C$-pseudo-cone in $\mathbb{R}^{n}$, for every Borel set $\eta\subset\Omega$, the Gaussian surface area measure of $K$ is then defined by
\begin{equation}
	S_{\gamma^{n}}(K,\eta)=\frac{1}{(\sqrt{2 \pi})^{n}} \int_{\nu_{K}^{-1}(\eta)} e^{-\frac{|x|^{2}}{2}} d \mathcal{H}^{n-1}(x).
\end{equation}

The Gaussian surface area measure of $C$-pseudo-cones is finite.
\begin{lemma}\label{finite surface}
	Let  $K$  be a  $C$-pseudo-cone, then the measure $S_{\gamma^{n}}(K,\cdot)$ is finite.
\end{lemma}
\begin{proof}
	Let $\Theta(x)=|x|^{-q}$ with some $q>n-1$, then there exists $t_{0}>0$ such that for any $x\in C^{+}(t_{0})$, we have $ e^{-\frac{|x|^{2}}{2}}\le \Theta(x)$. From \cite{schneider weighted} we know
	$$\int_{\partial K\cap C^{+}(t_{0})}\Theta(x)d\mathcal{H}^{n-1}(x)<\infty,$$
	so we get
	$$\int_{\partial K\cap C^{+}(t_{0})}e^{-\frac{|x|^{2}}{2}}d\mathcal{H}^{n-1}(x)<\infty.$$
	Since
	$$\int_{\partial K\cap C^{-}(t_{0})}e^{-\frac{|x|^{2}}{2}}d\mathcal{H}^{n-1}(x)<\infty,$$
	then the desired result follows.
		
\end{proof}
We now require the transform integral formula from $\partial K$ to $\Omega_{C}$, the following  result can be found in \cite[Lemma 4]{schneider weighted cone}.

\begin{lemma}\label{integral trans}
		Let  $K$  be a  $C$-pseudo-cone. Let $F: \partial_{i}K\to \mathbb{R}$ be nonnegative and Borel measurable or $\mathcal{H}^{n-1}$-integrable. Then
		\begin{align}
			\int_{\partial_{i} K} F(y) d\mathcal{H}^{n-1}(y) & =\int_{\Omega_{C}} F\left(r_{K}(v)\right) \frac{\varrho_{K}^{n}(v)}{\bar{h}_{K}\left(\alpha_{K}(v)\right)} d v \\
			& =\int_{\Omega_{C}} F\left(r_{K}(v)\right) \frac{\varrho_{K}^{n-1}(v)}{\left|\left\langle v, \alpha_{K}(v)\right\rangle\right|} d v.
		\end{align}
\end{lemma}

The following transform integral formula for the Gaussian surface area measure will be needed.
\begin{lemma}\label{surface trans}
	 Let  $\omega \subset \Omega$  be a nonempty Borel set and  $g: \omega \rightarrow \mathbb{R}$  be bounded and measurable. Then
	\begin{equation}
	\int_{\omega} g(u) dS_{\gamma^{n}}(K, u)=\frac{1}{(\sqrt{2 \pi})^{n}}\int_{\alpha_{K}^{-1}(\omega)} g\left(\alpha_{K}(v)\right)  e^{-\frac{\varrho_{K}^{2}(v)}{2}}\frac{\varrho_{K}^{n-1}(v)}{\left|\left\langle v, \alpha_{K}(v)\right\rangle\right|} d v.
	\end{equation}
\end{lemma}
	
	\begin{proof}
		Let  $F(y):=1_{\omega}\left(\nu_{K}(y)\right) g\left(\nu_{K}(y)\right) e^{-\frac{|y|^{2}}{2}}$, from Lemma \ref{integral trans}  we get
		$$
		\int_{\partial_{i} K} 1_{\omega}\left(\nu_{K}(y)\right) g\left(\nu_{K}(y)\right)e^{-\frac{|y|^{2}}{2}} d\mathcal{H}^{n-1}(y)=\int_{\alpha_{K}^{-1}(\omega)} g\left(\alpha_{K}(v)\right) e^{-\frac{\varrho_{K}^{2}(v)}{2}} \frac{\varrho_{K}^{n-1}(v)}{\mid\left\langle v, \alpha_{K}(v) \rangle \mid\right.} d v.$$
		By the definition of  Gaussian surface area measure   we have
		$$\int_{\omega} g(u) dS_{\gamma^{n}}(K,u) =\frac{1}{(\sqrt{2 \pi})^{n}}\int_{\partial_{i} K} 1_{\omega}\left(\nu_{K}(y)\right) g\left(\nu_{K}(y)\right) e^{-\frac{|y|^{2}}{2}} d\mathcal{H}^{n-1}(y),$$
		then the desired result follows.
		
	\end{proof}

The weak continuity of the Gaussian surface area measure in the sense of Definition \ref{conv def} is established as follows, which will be used for the Minkowski type problem on the set $\mathcal{K}(C, \omega)$ for some compact $\omega \subset \Omega$.

\begin{lemma}\label{weak conve}
	Let  $\omega \subset \Omega$  be a nonempty compact subset, and let $ K_{j} \in \mathcal{K}(C, \omega)$  for  $j \in   \mathbb{N}$. Then $K_{j} \rightarrow K $ as  $j \rightarrow \infty $ implies the weak convergence  $S_{\gamma^{n}}\left(K_{j}, \cdot\right) \xrightarrow{w} S_{\gamma^{n}}\left(K, \cdot\right)$.
\end{lemma}

\begin{proof}
	Let  $g: \Omega \rightarrow \mathbb{R}$  be bounded and continuous. By Lemma \ref{surface trans}  we obtain
$$	\int_{\Omega} g(u) dS_{\gamma^{n}}\left(K_{j}, u\right)=\frac{1}{(\sqrt{2 \pi})^{n}}\int_{\Omega_{C}} g\left(\alpha_{K_{j}}(v)\right) e^{-\frac{\varrho_{K_{j}}^{2}(v)}{2}} \frac{\varrho_{K_{j}}^{n-1}(v)}{\left|\left\langle v, \alpha_{K_{j}}(v)\right\rangle\right|} d v.$$

	For almost all  $v \in \Omega_{C}$  we have
	$$g\left(\alpha_{K_{j}}(v)\right) e^{-\frac{\varrho_{K_{j}}^{2}(v)}{2}} \frac{\varrho_{K_{j}}^{n-1}(v)}{\left|\left\langle v, \alpha_{K_{j}}(v)\right\rangle\right|} \rightarrow g\left(\alpha_{K}(v)\right) e^{-\frac{\varrho_{K}^{2}(v)}{2}} \frac{\varrho_{K}^{n-1}(v)}{\left|\left\langle v, \alpha_{K}(v)\right\rangle\right|},$$
as $j\to \infty$.
	Since $K_{j} \rightarrow K$ on the set $\mathcal{K}(C, \omega)$,  then $\left|\left\langle v, \alpha_{K_{j}}(v)\right\rangle\right|$ is bounded away from $0$ because $\alpha_{K_{j}}(v)\subset \omega$ and $\omega$ is compact,
	and $\varrho_{K_{j}} $ is uniformly bounded. Thus,  the functions on the left-hand side are uniformly bounded, and this result follows from the dominated convergence theorem.
\end{proof}
To study the Minkowski type existence problem by variational methods, we need the variational formula for the Gaussian measure of pseudo-cones. The following  result can be found in \cite{schneider weighted}. The similar result for convex bodies is presented in \cite{HLYZdual}.
\begin{lemma}\label{rho deri}
	 Let $ \omega \subset \Omega$  be nonempty and compact, let  $K \in \mathcal{K}(C, \omega)$. Let  $f: \omega \rightarrow \mathbb{R}$  be continuous. There is a constant  $\delta>0$  such that the function  $h_{t}$  defined by
	$$h_{t}(u):=\bar{h}_{K}(u)+t f(u), \quad u \in \omega$$
	is positive for  $|t| \leq \delta$. Let  $\left[h_{t}\right]$  be the Wulff shape associated with  $\left(C, \omega, h_{t}\right)$, for  $|t| \leq \delta$.
	
	(a) For almost all  $v \in \Omega_{C}$,
	
	$$\left.\frac{\mathrm{d} \varrho_{\left[h_{t}\right]}(v)}{\mathrm{d} t}\right|_{t=0}=\lim _{t \rightarrow 0} \frac{\varrho_{\left[h_{t}\right]}(v)-\varrho_{K}(v)}{t}=\frac{f\left(\alpha_{K}(v)\right)}{\bar{h}_{K}\left(\alpha_{K}(v)\right)} \varrho_{K}(v).$$
	
	(b) There is a constant  $M$  with
	
	 $$\left|\varrho_{\left[h_{t}\right]}(v)-\varrho_{K}(v)\right| \leq M|t|$$
	
	for all  $v \in \Omega_{C}$  and all  $|t| \leq \delta$.
\end{lemma}

For  convenience of calculations,  the Gaussian covolume $V_{G}(K) $ for $C$-pseudo-cone $K$ is defined by
\begin{align}\label{coco}
V_{G}(K)=\gamma^{n}(C \backslash K).
\end{align}

Then the variational formula for the Gaussian covolume is established.

\begin{lemma}\label{covari}
	Let  $K \in \mathcal{K}(C, \omega)$, for some nonempty, compact set  $\omega \subset \Omega$.  Let  $f: \omega \rightarrow \mathbb{R}$  be continuous, and let  $\left[\left.\bar{h}_{K}\right|_{\omega}+t f\right]$  be the Wulff shape associated with  $\left(C, \omega,\left.\bar{h}_{K}\right|_{\omega}+t f\right)$. Then
	\begin{equation}
	\lim _{t \rightarrow 0} \frac{V_{G}\left(\left[\left.\bar{h}_{K}\right|_{\omega}+t f\right]\right)-V_{G}(K)}{t}=\int_{\omega} f(u) dS_{\gamma^{n}}(K,  u).
	\end{equation}
	
\end{lemma}

\begin{proof}
	For  convenience we denote $h_{t}(u):=\bar{h}_{K}(u)+tf(u)$ for $u\in\omega$ and small enough $|t|$. From  (\ref{coco}) we have
	$$V_{G}\left(\left[h_{t}\right]\right)=\frac{1}{(\sqrt{2 \pi})^{n}}\int_{\Omega_{C}} \int_{0}^{\varrho_{\left[h_{t}\right]}(v)} e^{-\frac{r^{2}}{2}} r^{n-1} dr d v=\frac{1}{(\sqrt{2 \pi})^{n}}\int_{\Omega_{C}} F_{t}(v)dv,$$
	where $F_{t}(v)=\int_{0}^{\varrho_{\left[h_{t}\right]}(v)} e^{-\frac{r^{2}}{2}} r^{n-1} dr$.
	Then we obtain
	\begin{align*}
		\frac{F_{t}(v)-F_{0}(v)}{t} & =\frac{1}{t} \int_{\varrho_{K}(v)}^{\varrho_{\left[h_{t}\right]}(v)} e^{-\frac{r^{2}}{2}} r^{n-1} d r \\
		& =\frac{\varrho_{\left[h_{t}\right]}(v)-\varrho_{K}(v)}{t} \cdot \frac{1}{\varrho_{\left[h_{t}\right]}(v)-\varrho_{K}(v)} \int_{\varrho_{K}(v)}^{\varrho_{\left[h_{t}\right]}(v)} e^{-\frac{r^{2}}{2}} r^{n-1} d r.
	\end{align*}
From Lemma \ref{rho deri}, the first term converges to $\frac{f\left(\alpha_{K}(v)\right)}{\bar{h}_{K}\left(\alpha_{K}(v)\right)} \varrho_{K}(v)$ as $t\to 0$. The second term converges to $e^{-\frac{\varrho_{K}^{2}}{2}} \varrho_{K}^{n-1}$.  Then we get
	\begin{equation*}
		\lim _{t \rightarrow 0} \frac{F_{t}(v)-F_{0}(v)}{t}=e^{-\frac{\varrho_{K}^{2}}{2}} \frac{f\left(\alpha_{K}(v)\right)}{\bar{h}_{K}\left(\alpha_{K}(v)\right)} \varrho_{K}^{n}(v)
	\end{equation*}
	for almost all $v \in \Omega_{C}$. From Lemma \ref{rho deri} we have
	$$\left|\frac{F_{t}(v)-F_{0}(v)}{t}\right|\le M$$
	for some $M>0$ for sufficiently small $|t|$. By the dominated convergence theorem, $F(y)=f(\nu_{K}(y))e^{-\frac{|y|^{2}}{2}}$ in Lemma \ref{integral trans}, and the fact that  $S_{\gamma^{n}}(K,  \cdot)$ is concentrated on $\omega$  for any $K\in\mathcal{K}(C,\omega)$,  we can get
	\begin{align*}
		\lim _{t \rightarrow 0} \frac{V_{G}\left(\left[h_{t}\right]\right)-V_{G}(K)}{t}&=\frac{1}{(\sqrt{2 \pi})^{n}}\int_{\Omega_{C}}	\lim _{t \rightarrow 0} \frac{F_{t}(v)-F_{0}(v)}{t}dv\\
		&=\frac{1}{(\sqrt{2 \pi})^{n}}\int_{\Omega_{C}}e^{-\frac{\varrho_{K}^{2}}{2}} \frac{f\left(\alpha_{K}(v)\right)}{\bar{h}_{K}\left(\alpha_{K}(v)\right)} \varrho_{K}^{n}(v)dv\\
		&=\frac{1}{(\sqrt{2 \pi})^{n}}\int_{\partial_{i}K}f(\nu_{K}(y))e^{-\frac{|y|^{2}}{2}}d\mathcal{H}^{n-1}(y)\\
		&=\int_{\Omega}f(u)dS_{\gamma^{n}}(K,  u)\\
		&=\int_{\omega}f(u)dS_{\gamma^{n}}(K,  u).
		\end{align*}
\end{proof}

  The following Lemma shows the variational formula for the \textbf{real} Gaussian volume of $K \in \mathcal{K}(C, \omega)$.

 \begin{lemma}\label{vari formula}
 	Let  $K \in \mathcal{K}(C, \omega)$, for some nonempty, compact set  $\omega \subset \Omega$.  Let  $f: \omega \rightarrow \mathbb{R}$  be continuous, and let  $\left[\left.\bar{h}_{K}\right|_{\omega}+t f\right]$  be the Wulff shape associated with  $\left(C, \omega,\left.\bar{h}_{K}\right|_{\omega}+t f\right)$. Then
 	\begin{equation}
 		\lim _{t \rightarrow 0} \frac{\gamma^{n}\left(\left[\left.\bar{h}_{K}\right|_{\omega}+t f\right]\right)-\gamma^{n}(K)}{t}=-\int_{\omega} f(u) dS_{\gamma^{n}}(K,  u).
 	\end{equation}
 	
 \end{lemma}

\begin{proof}
	The following fact is true for any $C$-pseudo-cone $L$,
	$$\gamma^{n}(L)+V_{G}(L)=\gamma^{n}(C)<\frac{1}{2}.$$
	Then, by Lemma \ref{covari}, it follows that
		\begin{align*}
		\lim _{t \rightarrow 0} \frac{\gamma^{n}\left(\left[h_{t}\right]\right)-\gamma^{n}(K)}{t}&=	 \lim _{t \rightarrow 0} \frac{\left( \gamma^{n}(C)-V_{G}\left(\left[h_{t}\right]\right)\right) -\left( \gamma^{n}(C)-V_{G}(K)\right) }{t}\\
		&=-	\lim _{t \rightarrow 0} \frac{V_{G}\left(\left[h_{t}\right]\right)-V_{G}(K)}{t}\\
		&=-\int_{\omega} f(u) dS_{\gamma^{n}}(K,  u).
	\end{align*}
\end{proof}

The last Lemma in this section shows the continuity of Gaussian volume  as follows.
\begin{lemma}\label{continuity of Gaussian volume}
	The Gaussian volume functional $\gamma^{n}(\cdot)$ of C-pseudo-cones is continuous.
\end{lemma}

\begin{proof}
Let $K_{j}$ be a sequence of C-pseudo-cones, which converges to a C-pseudo-cone $K$, that is, $K_{j}\rightarrow K$.  Let $\epsilon>0$ be given, we can choose a sufficiently large $t$ such that $\gamma^{n}(C^{+}(t))<\frac{\epsilon}{2}$ and $K_{j}\cap C^{-}(t)\neq \emptyset$ for $j\in \mathbb{N}$. The convergence of C-pseudo-cones in Definition \ref{conv def} yields
	$$K_{j}\cap C^{-}(t)\rightarrow K\cap C^{-}(t).$$
	From the continuity of Gaussian volume for convex bodies in \cite{shana2024}, there exists $j_{0}$ such for all $j>j_{0}$,
	$$|\gamma^{n}(K_{j}\cap C^{-}(t))-\gamma^{n}(K\cap C^{-}(t))|<\frac{\epsilon}{2},$$
	then we have
	\begin{align*}
	&|\gamma^{n}(K_{j})-\gamma^{n}(K)|\\
	=&|\left( \gamma^{n}(K_{j}\cap C^{+}(t))+\gamma^{n}(K_{j}\cap C^{-}(t))\right) -\left( \gamma^{n}(K\cap C^{+}(t))+\gamma^{n}(K\cap C^{-}(t))\right) |\\
	=&|\left( \gamma^{n}(K_{j}\cap C^{+}(t))-\gamma^{n}(K\cap C^{+}(t))\right) +\left( \gamma^{n}(K_{j}\cap C^{-}(t))-\gamma^{n}(K\cap C^{-}(t))\right) |\\
	\le&| \gamma^{n}(K_{j}\cap C^{+}(t))-\gamma^{n}(K\cap C^{+}(t))| +| \gamma^{n}(K_{j}\cap C^{-}(t))-\gamma^{n}(K\cap C^{-}(t)) |\\
	\le&\gamma^{n}(C^{+}(t))+| \gamma^{n}(K_{j}\cap C^{-}(t))-\gamma^{n}(K\cap C^{-}(t)) |\\
	<&\frac{\epsilon}{2}+\frac{\epsilon}{2}=\epsilon.
	\end{align*}
	
Therefore,	$\gamma^{n}(K_{j})\rightarrow\gamma^{n}(K)$.
\end{proof}

\section{  Gaussian Minkowski problem for C-pseudo-cones}

The following existence theorem of solutions to the Gaussian  Minkowski problem will be established through a variational argument combined with approximation method. Notably, the variational techniques in homogeneous weighted Minkowski problem \cite{schneider weighted cone,schneider weighted} and the classical Gaussian Minkowski problem for convex bodies \cite{Huang2021} are not directly applicable in this setting.
So we introduce a new variational functional to solve this problem. Specifically, we employ the Gaussian volume instead of the covolume in the usual pseudo-cone theory.

\begin{theorem}\label{cone minkowski}
	Let $\mu$ be a nonzero, finite Borel measure on $\Omega$. Then there exists a $C$-pseudo-cone $K$ with
		$$cS_{\gamma^{n}}(K,\cdot)=\mu,$$
	where $c=\frac{\int_{\Omega}\bar{h}_{K}d\mu}{\gamma^{n}(K)}.$
\end{theorem}

Let $\mu$ be a nonzero, finite Borel measure on $\omega$ with some compact $\omega \subset \Omega$.
For $f\in C^{+}(\omega)$, where  $C^{+}(\omega)$ denotes the set of continuous functions $f:\omega\rightarrow (0,\infty)$, we define a functional $I_{\mu}:  C^{+}(\omega)\rightarrow (0,\infty)$ by
$$I_{\mu}(f):=\gamma^{n}([f])\int_{\omega} f d\mu.$$
Here $[f]$ denotes the Wulff shape associated with $(C,\omega,f)$.  This functional plays a central role in the variational approach to the Gaussian Minkowski problem for $C$-pseudo-cones. We consider the following optimization problem
$$\sup \{I_{\mu}(f) :f\in C^{+}(\omega)\}.$$

It is easy to get that  $I_{\mu}(\bar{h}_{K})=\gamma^{n}(K)\int_{\omega} \bar{h}_{K} d\mu>0$ for any $K\in \mathcal{K}(C,\omega)$.

\begin{lemma}\label{compact cone minkowski}
		Let $\mu$ be a  nonzero, finite Borel measure on $\omega$ with some nonempty compact $\omega \subset \Omega$. Then there exists a $C$-pseudo-cone $K\in\mathcal{K}(C,\omega)$ with
		 $$\frac{\int_{\omega}\bar{h}_{K}d\mu}{\gamma^{n}(K)}S_{\gamma^{n}}(K,\cdot)=\mu.$$
	
\end{lemma}
\begin{proof}

	For each $f\in C^{+}(\omega)$, by the definition of the Wulff shape, we have $\bar{h}_{[f]}\ge f$. Therefore
	$$I_{\mu}(f)=\gamma^{n}([f])\int_{\omega} f d\mu\le \gamma^{n}([f])\int_{\omega}\bar{h}_{[f]}  d\mu=I_{\mu}(\bar{h}_{[f]} ).$$
	This inequality implies that the supremum of $I_{\mu}(f)$ is achieved by support functions of $K(C,\omega)$.
	Let $\bar{h}_{K_{i}}$ be a maximizing sequence for
$I_{\mu}$ with $K_i\in \mathcal{K}(C,\omega)$, i.e.,
	$$\lim _{i \rightarrow \infty} I_{\mu}(\bar{h}_{K_{i}})=\sup \left\{I_{\mu}(f): f \in C^{+}\left(\omega\right)\right\}>0.$$
	Define
	$$r_{i}=\min\{r:rB\cap K_{i}\neq\emptyset \}.$$
If there is a subsequence, still denoted by $r_{i}$, such that $r_{i}\rightarrow\infty$. Note that the support function of $K_{i} $ satisfies
	$$\bar{h}_{K_{i}}\le r_{i},$$
	and the set inclusion
	$K_{i}\subset \mathbb{R}^{n}\backslash r_{i}B$ implies
\begin{align*}
\gamma^{n}(K_{i})\le \gamma^{n}(\mathbb{R}^{n}\backslash r_{i}B)=1-\gamma^{n}(r_{i}B).
\end{align*}
 Then we have the estimate as $r_{i}\rightarrow\infty$
	\begin{align*}
		 I_{\mu}(\bar{h}_{K_{i}})&=\gamma^{n}(K_{i})\int_{\omega}\bar{h}_{K_{i}}  d\mu\\
		&\le (1-\gamma^{n}(r_{i}B))\mu(\omega) r_{i}\\
		&=\mathbb{P}(|x|>r_{i})\mu(\omega) r_{i}\rightarrow 0,
	\end{align*}
where $\mathbb{P}$ denotes  the probability corresponding to the standard normal distribution in $\mathbb{R}^{n}$.	The   convergence to zero follows from the exponential decay of normal distribution. This contradicts the fact that $\lim _{i \rightarrow \infty} I_{\mu}(\bar{h}_{K_{i}})=\sup I_{\mu}>0$ mentioned above.

 On the other hand, if $r_{i}\rightarrow 0$,
then
$$\mathbb{P}(|x|>r_{i})\rightarrow 1,$$
 by above estimate we also have
 	\begin{align*}
 	 I_{\mu}(\bar{h}_{K_{i}})&=\gamma^{n}(K_{i})\int_{\omega}\bar{h}_{K_{i}}  d\mu\\
 	&\le\mathbb{P}(|x|>r_{i})\mu(\omega) r_{i}\rightarrow 0,
 \end{align*}
  a contradiction.

These cases imply that there exist two constants $m$ and $M$ such that $0<m<\operatorname{dist}(o,\partial K_{i})<M$. By Lemma \ref{select}, there exists a  subsequence, denoted still as $K_{i}$, such that $K_{i}\rightarrow K$
for some $K\in\mathcal{K}(C,\omega)$ (see Lemma \ref{compace conv}). In other words, we have
$$\lim _{i \rightarrow \infty} I_{\mu}(\bar{h}_{K_{i}})=I_{\mu}(\bar{h}_{K})=\sup \left\{I_{\mu}(f): f \in C^{+}\left(\omega\right)\right\}.$$

Now let $f:\omega\rightarrow \mathbb{R}$ be continuous, for sufficiently small $|t|$ such that the function $\bar{h}_{K}+tf\in  C^{+}(\omega)$, from Lemma \ref{vari formula}, we get
\begin{equation}
	0=\left.\frac{d}{d t}\right|_{t=0}I_{\mu}(\bar{h}_{K}+tf)= \gamma^{n}(K)\int_{\omega}fd\mu-\int_{\omega}fdS_{\gamma^{n}}(K,\cdot)\int_{\omega}\bar{h}_{K}d\mu.
\end{equation}
Thus,
	\begin{equation*}
 \gamma^{n}(K)\int_{\omega}fd\mu=\int_{\omega}fdS_{\gamma^{n}}(K,\cdot)\int_{\omega}\bar{h}_{K}d\mu
	\end{equation*}
	holds for all continuous $f$ on $\omega$, by Riesz representation theorem, we then get $$\mu=\frac{\int_{\omega}\bar{h}_{K}d\mu}{\gamma^{n}(K)}S_{\gamma^{n}}(K,\cdot),$$
the desired result follows.
\end{proof}

 For a general measure $\mu$ on $\Omega$, if we choose a sequence $\{\omega_{i}\}_{i\in\mathbb{N}}$  of compact subsets of $\Omega$ such that $\omega_{i}\subset \operatorname{int}\omega_{i+1}$, and $\cup_{j\in\mathbb{N}} \omega_{j}=\Omega$. For each $i\in\mathbb{N}$,  define the restricted measure $\mu_{i}=\mu\llcorner\omega_{i}$ by $$\mu\llcorner\omega_{i}(\sigma):=\mu(\sigma\cap\omega_{i})$$ for any Borel subset $\sigma\subset\Omega$. The following uniform estimate of $I_{\mu_{i}}(\cdot)$ will be required.

 \begin{lemma}\label{uniesti}
 	Let $\mu$ be a nonzero, finite Borel measure on $\Omega$, $\mu_{i}=\mu\llcorner\omega_{i}$, there exists a $i_{0}\in\mathbb{N}$ such that $\mu_{i_{0}}\neq0$. Let $$I_{\mu_{i}}(\bar{h}_{K_{i}})=\sup \left\{I_{\mu_{i}}(f)=\gamma^{n}([f])\int_{\omega_{i}} f d\mu_{i}: f \in C^{+}\left(\omega_{i}\right)\right\}$$ for $i\ge i_{0}$. Then there exists a constant $a>0$  such that $I_{\mu_{i}}(\bar{h}_{K_{i}})>a$ for any $i\ge i_{0}$.
 \end{lemma}
 \begin{proof}
 	For $i\ge i_{0}$, $K_{i}\in\mathcal{K}(C,\omega_{i})$, we have
 	$$K_{i}=C \cap \bigcap_{u \in \omega_{i}} H_{K_{i}}^{-}(u)\supset C \cap \bigcap_{u \in \omega_{i+1}} H_{K_{i}}^{-}(u)\supset C \cap \bigcap_{u \in \Omega} H_{K_{i}}^{-}(u)=K_{i},$$
 	therefore $K_{i}\in\mathcal{K}(C,\omega_{i+1})$, it follows that $K_{i}=[\bar{h}_{K_{i}}|_{\omega_{i}}]=[\bar{h}_{K_{i}}|_{\omega_{i+1}}]$. Since $\mu_{i}\le\mu_{i+1}\le\mu$, $\omega_{i}\subset \operatorname{int}\omega_{i+1}$, we obtain
 	\begin{align*}
 0<I_{\mu_{i}}(\bar{h}_{K_{i}})=&\gamma^{n}([\bar{h}_{K_{i}}|_{\omega_{i}}])\int_{\omega_{i}}\bar{h}_{K_{i}}  d\mu_{i}\\
 		 &=\gamma^{n}([\bar{h}_{K_{i}}|_{\omega_{i+1}}])\int_{\omega_{i}}\bar{h}_{K_{i}}  d\mu_{i}\\
 		&\le \gamma^{n}([\bar{h}_{K_{i}}|_{\omega_{i+1}}])\int_{\omega_{i+1}}\bar{h}_{K_{i}}  d\mu_{i+1}\\&\le I_{\mu_{i+1}}(\bar{h}_{K_{i+1}}).
 	\end{align*}
 	Therefore, $I_{\mu_{i}}(\bar{h}_{K_{i}})\ge I_{\mu_{i_{0}}}(\bar{h}_{K_{i_{0}}})>0$ for any  $i\ge i_{0}$.

 \end{proof}
We are now in a position to prove  Theorem \ref{cone minkowski} by using the approximation method.
\\
\textit{Proof of Theorem \ref{cone minkowski}.} 	Let $\{\omega_{i}\}_{i\in\mathbb{N}}$  be a sequence of compact subsets of $\Omega$ satisfying
$$\omega_{i}\subset \operatorname{int}\omega_{i+1}\quad \text{and}\quad \cup_{j\in\mathbb{N}} \omega_{j}=\Omega.$$ Let $\mu_{i}=\mu\llcorner\omega_{i}$, where $\mu\llcorner\omega_{i}(\sigma):=\mu(\sigma\cap\omega_{i})$ for any Borel subset $\sigma\subset\Omega$. There exists a $i_{0}$ such that $\mu_{i_{0}}\neq0$. By Lemma \ref{compact cone minkowski}, for each $i\in\mathbb{N}$, $i\ge i_{0}$, there exists a $C$-pseudo-cone $K_{i}\in\mathcal{K}(C,\omega_{i})$ satisfying
$$\frac{\int_{\omega_{i}}\bar{h}_{K_{i}}d\mu_{i}}{\gamma^{n}(K_{i})}S_{\gamma^{n}}(K_{i},\cdot)=\mu_{i}.$$

Let $r_{i}=\min\{r:rB\cap K_{i}\neq\emptyset \}$, then $K_{i}\subset \mathbb{R}^{n}\backslash r_{i}B$.
  If there is a subsequence, still denoted by $K_{i}$, such that $r_{i}\rightarrow\infty$,  	then
\begin{align*} I_{\mu_{i}}(\bar{h}_{K_{i}})
	&\le (1-\gamma^{n}(r_{i}B))\mu_{i}(\omega_{i}) r_{i}\\
	&\le (1-\gamma^{n}(r_{i}B))\mu(\Omega) r_{i}\\
	&=\mathbb{P}(|x|>r_{i})\mu(\Omega) r_{i}\rightarrow 0.
\end{align*}
This contradicts the strictly positive lower bound $	 0<I_{\mu_{i}}(\bar{h}_{K_{i}})\le I_{\mu_{i+1}}(\bar{h}_{K_{i+1}})\le \cdots$ for $i\ge i_{0}$ proved in Lemma \ref{uniesti}.

If $r_{i}\rightarrow0$,
 we also have
\begin{align*} I_{\mu_{i}}(\bar{h}_{K_{i}})\le\mathbb{P}(|x|>r_{i})\mu(\Omega) r_{i}\rightarrow 0,
\end{align*}
a contradiction.

Therefore, there exist two constants $m$ and $M$ such that $0<m<\operatorname{dist}(o,\partial K_{i})<M$. From  Lemma \ref{select}, there exists a subsequence of $K_{i}\in\mathcal{K}(C,\omega_{i})$,  denoted still  as $K_{i}$, such that $K_{i}\rightarrow K$.

 Since $K_{i}$ are $C$-pseudo-cones, then there exists $z\in K_{i}$ for all $i\in\mathbb{N}$ such that $z+C\subset K_{i}$, i.e., $\gamma^{n}(K_{i})\ge\gamma^{n}(z+C)$. The normalization constants $$c_{i}:=\frac{\int_{\omega_{i}}\bar{h}_{K_{i}}d\mu_{i}}{\gamma^{n}(K_{i})}$$ are uniformly bounded, it then follows Lemma \ref{continuity of Gaussian volume} that
$c_{i}\rightarrow c=\frac{\int_{\Omega}\bar{h}_{K}d\mu}{\gamma^{n}(K)}$ as $i\to\infty$.

Fix an index $i$, choose a compact set $\beta\subset\Omega$ such that $\omega_{i}\subset\operatorname{int}\beta$, then it follows Lemma 13 in \cite{schneider weighted},
$$\nu_{K_{j}}^{-1}(\omega_{i})=\nu_{K_{j}^{(\beta)}}^{-1}(\omega_{i}),\quad \nu_{K}^{-1}(\omega_{i})=\nu_{K^{(\beta)}}^{-1}(\omega_{i}),$$
and from Lemma \ref{compace conv}, we have
$$K_{i}^{(\beta)}\rightarrow K^{(\beta)}.$$
Lemma \ref{weak conve} shows that the $S_{\gamma^{n}}(K_{j})\llcorner\omega_{i}$ converges to $S_{\gamma^{n}}(K^{(\beta)})\llcorner\omega_{i}$ weakly. Then we have
$$c_{j}S_{\gamma^{n}}(K_{j},\cdot)\llcorner\omega_{i}\xrightarrow{w}cS_{\gamma^{n}}(K,\cdot)\llcorner\omega_{i}.$$
Since $c_{i}S_{\gamma^{n}}(K_{i},\cdot)\llcorner\omega_{i}=\mu\llcorner\omega_{i}$, it follows that for each Borel set $\sigma\subset\omega_{i}$ we have $c S_{\gamma^{n}}(K,\sigma)=c S_{\gamma^{n}}(K^{(\beta)},\sigma)=\mu_{i}(\sigma)=\mu(\sigma)$.
Since $\cup_{j\in\mathbb{N}} \omega_{j}=\Omega$, we conclude
$cS_{\gamma^{n}}(K,\cdot)=\mu$.
\\

\section{Uniqueness results for solutions }
The Ehrhard inequality was shown by Ehrhard  \cite{ehrhard} and generalized to different versions \cite{Latalaicm}. It implies the Gaussian isoperimetric inequality that among all subsets
of $ \mathbb{R}^{n} $ with prescribed Gaussian measure, half-spaces have the least
Gaussian perimeter. For $C$-pseudo-cones (convex sets),  the Ehrhard inequality with equality condition is posed by Shenfeld-Handel in \cite{shenfeld2018}
as follows.

\begin{theorem}
Let  $K, L$  be two $C$-pseudo-cones in  $\mathbb{R}^{n}$. For  $0<t<1$, we have
\begin{equation}\label{EHR}
\Phi^{-1}\left(\gamma^{n}((1-t) K+t L)\right) \geq(1-t) \Phi^{-1}\left(\gamma^{n}(K)\right)+t \Phi^{-1}\left(\gamma^{n}(L)\right).
\end{equation}
Here,
$$\Phi(x)=\frac{1}{\sqrt{2 \pi}} \int_{-\infty}^{x} e^{-\frac{t^{2}}{2}} d t.$$
Moreover, equality holds if and only if  $K=L$.
\end{theorem}
General versions of the Ehrhard inequality were introduced  in \cite{shenfeld2018}. The  log concavity property of   Gaussian measures, as stated below, can be deduced from Ehrhard inequality.

\begin{lemma}\label{log cancave}
 Let  $K, L$  be two $C$-pseudo-cones in  $\mathbb{R}^{n}$. For  $0<t<1$, we have
	\begin{equation}
		\gamma^{n}((1-t) K+t L) \geq \gamma^{n}(K)^{1-t} \gamma^{n}(L)^{t}
	\end{equation}
	with equality if and only if $K=L$.
\end{lemma}

From the variational formula showed  in Lemma \ref{vari formula}, the following Minkowski type inequality for $C$-pseudo-cones is established.
\begin{lemma}\label{Minkowski ine}
	If $K,L\in\mathcal{K}(C,\omega)$ for some nonempty compact $\omega\subset\Omega$, we have
	\begin{equation}\label{Minkowski type ineq pse}
 \frac{1}{\gamma^{n}(K)}\int_{\omega}
\left(\bar{h}_{K}-\bar{h}_{L}\right)dS_{\gamma^{n}}(K,u)\ge \log\gamma^{n}(L)-\log\gamma^{n}(K),
	\end{equation}
	with equality if and only if $K=L$.
\end{lemma}
\begin{proof}
	From the definition of the Wulff shape \eqref{wulff shape}, for  $0<t<1$, we have
	\begin{align*} [\bar{h}_{K}|_{\omega}+t(\bar{h}_{L}|_{\omega}-\bar{h}_{K}|_{\omega})]&=C \cap \bigcap_{u \in \omega}\left\{y \in \mathbb{R}^{n}:\langle y, u\rangle \leq (1-t)h_{K}(u)+th_{L}(u)\right\}\\
		&\supset C \cap \bigcap_{u \in \Omega}\left\{y \in \mathbb{R}^{n}:\langle y, u\rangle \leq (1-t)h_{K}(u)+th_{L}(u)\right\}\\
		&=(1-t)K+tL.
	\end{align*}
	If $K$ is a $C$-pseudo-cone, for any $a>0$, we have $aK\subset C$ and $aK+C=aK+aC=a(K+C)\subset aK$, hence $aK$ is also a $C$-pseudo-cone, then the last equality follows from \cite{sz2024asyp,Wang2024asymptotic}. Together with Lemma \ref{log cancave}, this yields $$\gamma^{n}([\bar{h}_{K}|_{\omega}+t(\bar{h}_{L}|_{\omega}-\bar{h}_{K}|_{\omega})])\geq\gamma^{n}((1-t) K+t L) \geq \gamma^{n}(K)^{1-t} \gamma^{n}(L)^{t}.$$
This formula can be formulated as
	\begin{align*}
	\log\left( \gamma^{n}([\bar{h}_{K}|_{\omega}+t(\bar{h}_{L}|_{\omega}-\bar{h}_{K}|_{\omega})])\right) &\geq\log\left( \gamma^{n}((1-t) K+t L)\right)\\
	&  \geq (1-t)\log \gamma^{n}(K) +t\log \gamma^{n}(L).
	\end{align*}
	Since $K\in\mathcal{K}(C,\omega)$, when $t=0$, then $\gamma^{n}([\bar{h}_{K}|_{\omega}])=\gamma^{n}(K)$. Let
\begin{align*}
f(t)=\log\left( \gamma^{n}([\bar{h}_{K}|_{\omega}+t(\bar{h}_{L}|_{\omega}-\bar{h}_{K}|_{\omega})])\right)   - (1-t)\log \gamma^{n}(K) -t\log \gamma^{n}(L),
\end{align*}
for  $0\le t \le 1$, it follows that  $f(0)=0$ and $f(t)\ge 0$, then the right derivative of $f(t)$ at $t=0$ is nonnegative, that is, $f'_+(0)\ge 0$.
	From Lemma \ref{vari formula}, we get $$-\frac{1}{\gamma^{n}(K)}\int_{\omega}
\left(\bar{h}_{L}-\bar{h}_{K}\right)dS_{\gamma^{n}}(K,u)\ge \log\gamma^{n}(L)-\log\gamma^{n}(K).$$
	If equality holds in \eqref{Minkowski type ineq pse}, we have
	\begin{align}\label{num}
	 	&\lim _{t \rightarrow 0^{+}} \frac{\log\left( \gamma^{n}\left([\bar{h}_{K}|_{\omega}+t(\bar{h}_{L}|_{\omega}-\bar{h}_{K}|_{\omega})]\right)\right) -\log\left( \gamma^{n}(K)\right) }{t}\\  \nonumber
	 	&=\lim _{t \rightarrow 0^{+}}\frac{\log\left( \gamma^{n}((1-t) K+t L)\right)-\log\left( \gamma^{n}(K)\right)}{t}\\  \nonumber
	 	&=\log\gamma^{n}(L)-\log\gamma^{n}(K).
	\end{align}
	Lemma \ref{log cancave} implies the function $g(t)=\log\left(\gamma^{n}((1-t) K+t L)\right) $ is concave in $[0,1]$. If equality holds, $g'(0)=g(1) -g(0)$ in \eqref{num}, then $g(t)$ is a linear function. Equality holds in Lemma \ref{log cancave} implies $K=L$.
\end{proof}

Now we define the Gaussian mixed volume $\gamma_{1}(K,L)$ for $C$-pseudo-cones.

\begin{definition}\label{mixvolume}
	For  $C$-pseudo-cones $ K ,L$,
	the \emph{Gaussian  mixed volume} is defined by
	\begin{equation}
		\gamma_{1}(K,L)= \int_{\Omega} \bar h_{L}(u) d S_{\gamma^{n}}(K, u).
	\end{equation}
\end{definition}

	Then Lemma \ref{Minkowski ine} implies
	\begin{equation}\label{mixed MInkowski}
		 \gamma_{1}(K,K)-\gamma_{1}(K,L)\ge\gamma^{n}(K)\log\frac{\gamma^{n}(L)}{\gamma^{n}(K)}.
	\end{equation}
	
	The uniqueness of the Gaussian Minkowski problem can be established when we restrict to the set of $\mathcal{K}(C,\omega)$.

\begin{theorem}\label{unique}
	Suppose $K,L\in\mathcal{K}(C,\omega)$ for some nonempty compact set $\omega\subset\Omega$, $\gamma^{n}(K)=\gamma^{n}(L)$. For any constant $c>0$, if
	$$cS_{\gamma^{n}}(K,\cdot)=S_{\gamma^{n}}(L,\cdot),$$
	 then $K=L$.
\end{theorem}

  \begin{proof}
  	From Lemma \ref{Minkowski ine} and \eqref{mixed MInkowski}, if
  	$cS_{\gamma^{n}}(K,\cdot)=S_{\gamma^{n}}(L,\cdot)$ and $\gamma^{n}(K)=\gamma^{n}(L)$, we have
  	\begin{align*}
  		\gamma_{1}(L,L)&=\int_{\Omega} \bar h_{L} d S_{\gamma^{n}}(L,u)=c\int_{\Omega} \bar h_{L} d S_{\gamma^{n}}(K,u)\\
  		&=c\gamma_{1}(K,L)\le c\gamma_{1}(K,K).
  	\end{align*}
  	Switching the role of $K$ and $L$, we have
  	\begin{align*}
  		\gamma_{1}(K,K)&=\int_{\Omega} \bar h_{K} d S_{\gamma^{n}}(K,u)=\frac{1}{c}\int_{\Omega} \bar h_{K} d S_{\gamma^{n}}(L,u)\\
  		&=\frac{1}{c}\gamma_{1}(L,K)\le \frac{1}{c}\gamma_{1}(L,L).
  	\end{align*}
  	Combining  two inequalities above we get
  	$$\gamma_{1}(L,L)\le c\gamma_{1}(K,K)\le c\left( \frac{1}{c}\gamma_{1}(L,L)\right) =\gamma_{1}(L,L).$$
  	Hence, from
  	Lemma \ref{Minkowski ine}, the equality conditions in above two inequalities imply $K=L$.
  \end{proof}

  In Lemma \ref{compact cone minkowski}, if there exist $K,L$ such that  $\gamma^{n}(K)=\gamma^{n}(L)$ and $\mu=cS_{\gamma^{n}}(K,\cdot)=c'S_{\gamma^{n}}(L,\cdot)$ for some constants $c,c'>0$,  then Theorem \ref{unique} yields $K=L$.

  However, if we remove the condition $\gamma^{n}(K)=\gamma^{n}(L)$ in Theorem \ref{unique}, then the uniqueness of the solution does not hold because  of
 the distinct decay rate of Gaussian measures.

\begin{theorem}\label{ex1}
	For any pointed, $n$-dimensional closed convex cone $C$ in $\mathbb{R}^{n}$ and any nonempty compact  $\omega\subset\Omega$, 	there exist $K,L\in\mathcal{K}(C,\omega)$,  such that $S_{\gamma^{n}}(K,\cdot)=S_{\gamma^{n}}(L,\cdot)$, but $K\neq L$.
\end{theorem}
\begin{proof}
	For any pointed, $n$-dimensional closed convex cone $C$ in $\mathbb{R}^{n}$ and any compact $\omega\subset\Omega$,
 since Gaussian measure is  rotationally invariant, then 	 for each $b\in\omega$ we can choose $b=-e_{n}\in\Omega$. Denote $H(t)=C\cap H(e_{n},t)$  and
	$$A=\{x\in\mathbb{R}^{n-1}:(x,1)\in H(1)\}.$$
	Let $g(t)=\int_{H(t)}e^{-\frac{|y|^{2}}{2}}d\mathcal{H}^{n-1}(y)$ for $t>0$, then $g(t)$ is Gaussian surface area of $H(t)$ (up to a constant).
	Then
	\begin{align*}
		 g(t)&=\int_{H(t)}e^{-\frac{|y|^{2}}{2}}d\mathcal{H}^{n-1}(y)\\
		&=\int_{(x,t)\in H(t)}e^{-\frac{|x|^{2}+t^{2}}{2}}d\mathcal{H}^{n-1}(x)\\
		 &=e^{-\frac{t^{2}}{2}}\int_{tA}e^{-\frac{|x|^{2}}{2}}d\mathcal{H}^{n-1}(x).
	\end{align*}
	 Therefore, we have $g(t)>0$ and $g(t)\rightarrow 0$ as $t\rightarrow 0$ or  $t\rightarrow\infty$.  Then there exist $t_{1},t_{2}>0$, $t_{1}\neq t_{2}$ such that $g(t_{1})=g(t_{2})$.
	
	Let $K=H(t_{1})+C$ and $L=H(t_{2})+C$, then $K$ and $L$ belong to $\mathcal{K}(C,b)$.  We get $$S_{\gamma^{n}}(K,b)=\frac{1}{(\sqrt{2\pi})^{n}}g(t_{1})$$
	and $$S_{\gamma^{n}}(L,b)=\frac{1}{(\sqrt{2\pi})^{n}}g(t_{2}),$$
the measures	$S_{\gamma^{n}}(K,\cdot)$ and $S_{\gamma^{n}}(L,\cdot)$ both have support in $b$. Therefore 	 $S_{\gamma^{n}}(K,\cdot)= S_{\gamma^{n}}(L,\cdot)$.
	
	For any $E\in \mathcal{K}(C,b)$, we have
	$$E=C \cap  H_{E}^{-}(b)\supset C \cap \bigcap_{u \in \omega} H_{E}^{-}(u)\supset C \cap \bigcap_{u \in \Omega} H_{E}^{-}(u)=E,$$
	hence $E=C \cap \bigcap_{u \in \omega} H_{E}^{-}(u)\in \mathcal{K}(C,\omega)$. Therefore, $K,L\in \mathcal{K}(C,\omega)$, the desired result follows.
	
\end{proof}

From the above discussion, we can also deduce that the assertion in Problem \ref{pro}  is false for any cone.
\begin{remark}
		Let $C\subset \mathbb{R}^{n}$ be a pointed, $n$-dimensional closed convex cone.	For each $b\in\Omega$,   after applying a suitable rotation, we choose $b=-e_{n}$. Define  $H(t)=C\cap H(b, -t)$ for $t>0$.
	From Theorem \ref{ex1}, we know that $g(t)=\int_{H(t)}e^{-\frac{|y|^{2}}{2}}d\mathcal{H}^{n-1}(y)$ is bounded, that is, $g(t)\le   c$ for some constant $c>0$, it follows that $S_{\gamma^{n}}(K,b)=\frac{1}{(\sqrt{2\pi})^{n}}g(t)
\le\frac{c}{(\sqrt{2\pi})^{n}}$ for $K=H(t)+C$. Therefore, \textbf{for any Borel measure $\mu$ on $\Omega$, if
	$\mu(b)>\frac{c}{(\sqrt{2\pi})^{n}}$,  then there does not exist a $C$-pseudo-cone $K$ such that $S_{\gamma^{n}}(K,\cdot)=\mu$.}
	
	For example, let  $C=\{ (r\cos\theta,r\sin\theta):r\ge 0, \theta\in[\alpha,\beta]\}$, where $0<\alpha<\beta<\pi$, then
	\begin{align*}
	&\frac{1}{(\sqrt{2\pi})^{2}}\int_{C\cap H(b,-t)}e^{-\frac{|y|^{2}}{2}}d\mathcal{H}^{1}(y)\\
	\le& \frac{1}{(\sqrt{2\pi})^{2}} \int_{\partial C}e^{-\frac{|y|^{2}}{2}}d\mathcal{H}^{1}(y)\\
	=&\frac{1}{\sqrt{2\pi}}\cdot \frac{1}{\sqrt{2\pi}} \int_{\partial C}e^{-\frac{|y|^{2}}{2}}d\mathcal{H}^{1}(y)\\
	=& \frac{1}{\sqrt{2\pi}},
	\end{align*}
	  where $$\frac{1}{\sqrt{2\pi}} \int_{\partial C}e^{-\frac{|y|^{2}}{2}}d\mathcal{H}^{1}(y)=\frac{1}{\sqrt{2\pi}} \int_{\mathbb{R}^{1} }e^{-\frac{|y|^{2}}{2}}d\mathcal{H}^{1}(y)=1$$
follows from the rotation invariance in Gaussian space. Hence, if $\mu(b)>\frac{1}{\sqrt{2\pi}}$, there does not exist a $C$-pseudo-cone $K$ such that $S_{\gamma^{n}}(K,\cdot)=\mu$.
\end{remark}

Theorem \ref{ex1} shows that  the assumption $\gamma^{n}(K)=\gamma^{n}(L)$ in Theorem \ref{unique} can't be removed, but the following  special example in $\mathbb{R}^{2}$ implies that it is possible to relax this assumption and still obtain the uniqueness result.
\begin{remark}
	Let
	 $$C=\left\{ (r\cos\theta,r\sin\theta):r\ge 0, \theta\in\left[\frac{\pi}{4},\frac{3\pi}{4}\right]\right\}\subset \mathbb{R}^{2}.$$ For each $v\in\Omega$, denote $C\cap H(-v,t)$ as $H(t)$.  Let $K_{t}=H(t)+C$, $t>0$,  and let the set $K_{t}^{'}$ be the convex hull of $H(t)$ and $-H(t)$, i.e.,
	 $$K_{t}^{'}=\operatorname{conv}\{H(t), -H(t)\}.$$
Then	$K_{t}^{'}$ is a convex body containing the origin in its interior. Specifically, $K_{t}^{'}$ is formed by the union of four congruent triangles. If $K\in \mathcal{K}(C,v)$, $K=K_{t}$, we denote $K_{t}^{'}$ by $K^{'}$.
	  Thus, for $K, L\in \mathcal{K}(C,v)$, we have $$S_{\gamma^{2}}(K,\cdot)=S_{\gamma^{2}}(L,\cdot) \quad\text{if and only if}\quad S_{\gamma^{2}}(K^{'},\cdot)=S_{\gamma^{2}}(L^{'},\cdot).$$
	From \cite{Huang2021}, if $\gamma^{2}(K^{'}), \gamma^{2}(L^{'})\ge \frac{1}{2}$ and $S_{\gamma^{2}}(K^{'},\cdot)=S_{\gamma^{2}}(L^{'},\cdot)$, then $K^{'}=L^{'}$. Since $\gamma^{2}(C)=\frac{1}{4}$ and $$\gamma^{2}(K^{'})=4\gamma^{2}(\operatorname{conv}\{H(t), o\})=4(\gamma^{2}(C)-\gamma^{2}(K)),$$
	the volume constraint  $\gamma^{2}(K^{'}), \gamma^{2}(L^{'})\ge \frac{1}{2}$ is equivalent to
	   $$\gamma^{2}(C\setminus K), \gamma^{2}(C\setminus L)\ge \frac{1}{8}=\frac{1}{2}\gamma^{2}(C).$$
In other words, $\gamma^{2}(K), \gamma^{2}( L)\le \frac{1}{2}\gamma^{2}(C)$. Under this setting, if  $S_{\gamma^{2}}(K^{'},\cdot)=S_{\gamma^{2}}(L^{'},\cdot)$,  we have $K^{'}=L^{'}$,  which implies $K=L$.
	
	In summary, for any $v\in\Omega$,  $K, L\in \mathcal{K}(C,v)$, if $S_{\gamma^{2}}(K,\cdot)=S_{\gamma^{2}}(L,\cdot)$ and  $\gamma^{2}(K), \gamma^{2}( L)\le \frac{1}{2}\gamma^{2}(C)$, then $K=L$.
\end{remark}

\section{Gaussian cone measure }

In this section, we discuss some key properties of the Gaussian cone measure.
The \textit{Gaussian cone measure} $	 C_{\gamma^{n}}(K,\cdot)$ is  defined as follows.
\begin{definition}\label{conedef}
	For every Borel set $\eta\subset\Omega$, the \textit{Gaussian cone  measure} of a C-pseudo-cone $K$ is given by
	\begin{equation}
		C_{\gamma^{n}}(K,\eta):=\frac{1}{(\sqrt{2 \pi})^{n}} \int_{\nu_{K}^{-1}(\eta)}|\langle x,\nu_{K}(x) \rangle| e^{-\frac{|x|^{2}}{2}} d \mathcal{H}^{n-1}(x)=\int_{\eta}\bar{h}_{K}(u)d	 S_{\gamma^{n}}(K,u).
	\end{equation}
\end{definition}
Geometrically,  $C_{\gamma^{n}}(K,\cdot)$  represents the product of the height (via the support function) and the Gaussian surface area. However, due to the inhomogeneity of the Gaussian measure, the $C_{\gamma^{n}}(K,\eta)$ does not equal  the true Gaussian cone volume $\gamma^{n}\left( \bigcup_{x \in \nu_{K}^{-1}(\eta)} [o, x]\right)  $  . As shown below, the Gaussian cone measure is strictly less than the corresponding Gaussian cone volume.


\begin{lemma}\label{volume ine}
	Let  $K$  be a  $C$-pseudo-cone, for every Borel set $\omega\subset\Omega$ we have
	$$\frac{1}{n}C_{\gamma^{n}}(K,\omega)<\gamma^{n}\left( \bigcup_{x \in \nu_{K}^{-1}(\omega)} [o, x]\right).$$
\end{lemma}

\begin{proof}
	Let $g(u)=\bar{h}_{K}(u)$ in Lemma \ref{surface trans}, then $\bar{h}_{K}(\alpha_{K}(u))= \left|\langle u\varrho_{K}(u),\alpha_{K}(u)\rangle\right|$. From Lemma \ref{surface trans} and Definition \ref{conedef}, we obtain
	\begin{equation}
		C_{\gamma^{n}}(K,\omega)=\int_{\omega}\bar{h}_{K}(u)d	 S_{\gamma^{n}}(K,u)=\frac{1}{(\sqrt{2 \pi})^{n}} \int_{\alpha_{K}^{-1}(\omega)} e^{-\frac{\varrho_{K}(u)^{2}}{2}}\varrho_{K}^{n}(u) d u,
	\end{equation}
	and
	\begin{align*}
		\gamma^{n}\left( \bigcup_{x \in \nu_{K}^{-1}(\omega)} [o, x]\right)&=\frac{1}{(\sqrt{2 \pi})^{n}}\int_{\bigcup_{x \in \nu_{K}^{-1}(\omega)} [o, x]}e^{-\frac{|x|^{2}}{2}}dx\\
		&=\frac{1}{(\sqrt{2 \pi})^{n}}\int_{\alpha_{K}^{-1}(\omega)}\int_{0}^{\varrho_{K}(u)}e^{-\frac{r^{2}}{2}}r^{n-1}drdu\\
		&>\frac{1}{(\sqrt{2 \pi})^{n}}\int_{\alpha_{K}^{-1}(\omega)}e^{-\frac{\varrho_{K}(u)^{2}}{2}}\int_{0}^{\varrho_{K}(u)}r^{n-1}drdu\\
		&=\frac{1}{n}\frac{1}{(\sqrt{2 \pi})^{n}} \int_{\alpha_{K}^{-1}(\omega)} e^{-\frac{\varrho_{K}(u)^{2}}{2}}\varrho_{K}^{n}(u) d u\\
		&=\frac{1}{n}C_{\gamma^{n}}(K,\omega).
	\end{align*}
\end{proof}
Let $\omega=\Omega$ in Lemma \ref{volume ine}, from Definition \ref{mixvolume} of Gaussian mixed volume, we have the following mixed volume inequality, which is analogous to Lemma 3.2 in \cite{shana2024}.

\begin{lemma}
		Let  $K$  be a  $C$-pseudo-cone, then
		$$\frac{1}{n}\gamma_{1}(K,K)<V_{G}(K).$$
\end{lemma}
Next, we require some key properties of the Gaussian cone measure.

\begin{lemma}
	Let  $K$  be a  $C$-pseudo-cone, the measure $C_{\gamma^{n}}(K,\cdot)$ is finite.
\end{lemma}
\begin{proof}
	This follows from the finiteness of the Gaussian surface area measure  $S_{\gamma^{n}}(K,\cdot)$ in Lemma \ref{finite surface} and the boundedness of the support functions of $C$-pseudo-cones.
\end{proof}

\begin{lemma}\label{weal cone conve}
	Let  $\omega \subset \Omega$  be a nonempty compact subset, and let $ K_{j} \in \mathcal{K}(C, \omega)$  for  $j \in   \mathbb{N}$. Then $K_{j} \rightarrow K $ as  $j \rightarrow \infty $ implies the weak convergence  $C_{\gamma^{n}}\left(K_{j}, \cdot\right) \xrightarrow{w} C_{\gamma^{n}}\left(K, \cdot\right)$.
\end{lemma}

\begin{proof}
	This follows from Lemma \ref{weak conve} and  the fact $\bar{h}_{K_{j}}\to\bar{h}_{K}$ uniformly.
\end{proof}

The following  Lemma provided in \cite{schneider weighted} will be needed.
\begin{lemma}\label{rho derilog}
	Let $ \omega \subset \Omega$  be a nonempty  compact set, and  $K \in \mathcal{K}(C, \omega)$. Consider a continuous function  $f: \omega \rightarrow \mathbb{R}$.   There exists a  sufficiently small  constant  $\delta>0$  such that the function  $h_{t}$  defined by
	\begin{equation}\label{logminkowskisum}
		\log h_{t}(u)=\log \bar{h}_{K}(u)+t f(u)+o(t, u), \quad u \in \omega
	\end{equation}
	for each $|t| \leq \delta$, where the function  $o(t, \cdot): \omega \rightarrow \mathbb{R}$  is continuous and   $\lim _{t \rightarrow 0} o(t, \cdot) / t=0$  uniformly on  $\omega$. Let  $\left[h_{t}\right]$  be the  Wulff shape associated with  $\left(C, \omega, h_{t}\right)$, we call $\left[h_{t}\right]$ \textit{a logarithmic
		family of Wulff shapes}.
	
	For almost all  $v \in \Omega_{C}$, we have
	\begin{equation}\label{log rho}
		\left.\frac{\mathrm{d} \varrho_{\left[h_{t}\right]}(v)}{\mathrm{d} t}\right|_{t=0}=\lim _{t \rightarrow 0} \frac{\varrho_{\left[h_{t}\right]}(v)-\varrho_{K}(v)}{t}=f\left(\alpha_{K}(v)\right)\varrho_{K}(v).
	\end{equation}
\end{lemma}

With the above preparations, we now establish the variational formula for the Gaussian covolume with respect to a logarithmic family of Wulff shapes.

\begin{lemma}\label{covarilog}
	Let  $K \in \mathcal{K}(C, \omega)$ for some nonempty, compact set  $\omega \subset \Omega$.   Consider a continuous function   $f: \omega \rightarrow \mathbb{R}$,  and let $\left[h_{t}\right]$  be the  Wulff shapes associated with  $\left(C, \omega, h_{t}\right) $, where
	$$\log h_{t}(u)=\log \bar{h}_{K}(u)+t f(u)+o(t, u).$$ Then
	\begin{equation*}
		\lim _{t \rightarrow 0} \frac{V_{G}\left(\left[h_{t}\right]\right)-V_{G}(K)}{t}=\int_{\omega} f(u) dC_{\gamma^{n}}(K,  u).
	\end{equation*}
	
\end{lemma}

\begin{proof}
	From (\ref{coco}), we have
	$$V_{G}\left(\left[h_{t}\right]\right)=\frac{1}{(\sqrt{2 \pi})^{n}}\int_{\Omega_{C}} \int_{0}^{\varrho_{\left[h_{t}\right]}(v)} e^{-\frac{r^{2}}{2}} r^{n-1} dr d v=\frac{1}{(\sqrt{2 \pi})^{n}}\int_{\Omega_{C}} F_{t}(v)dv,$$
	where $F_{t}(v)=\int_{0}^{\varrho_{\left[h_{t}\right]}(v)} e^{-\frac{r^{2}}{2}} r^{n-1} dr$.
	Then we get
	\begin{align*}
		\frac{F_{t}(v)-F_{0}(v)}{t} & =\frac{1}{t} \int_{\varrho_{K}(v)}^{\varrho_{\left[h_{t}\right]}(v)} e^{-\frac{r^{2}}{2}} r^{n-1} d r \\
		& =\frac{\varrho_{\left[h_{t}\right]}(v)-\varrho_{K}(v)}{t} \cdot \frac{1}{\varrho_{\left[h_{t}\right]}(v)-\varrho_{K}(v)} \int_{\varrho_{K}(v)}^{\varrho_{\left[h_{t}\right]}(v)} e^{-\frac{r^{2}}{2}} r^{n-1} d r.
	\end{align*}
As $t\to 0$,   the first term converges to $f\left(\alpha_{K}(v)\right) \varrho_{K}(v) $ by Lemma \ref{rho derilog}, while the second term converges to $e^{-\frac{\varrho_{K}^{2}}{2}} \varrho_{K}^{n-1}$. Combining these results we obtain
	\begin{equation*}
		\lim _{t \rightarrow 0} \frac{F_{t}(v)-F_{0}(v)}{t}=e^{-\frac{\varrho_{K}^{2}}{2}} f\left(\alpha_{K}(v)\right) \varrho_{K}^{n}(v)
	\end{equation*}
	for almost all $v \in \Omega_{C}$. Furthermore, for sufficiently small $|t|$, there exists a constant $M$ such that
	$$\left|\frac{F_{t}(v)-F_{0}(v)}{t}\right|\le M.$$
	By the dominated convergence theorem, along with the $g(u)=f(u)\bar{h}_{K}(u)$ in Lemma \ref{surface trans}, and the fact that  $C_{\gamma^{n}}(K,  \cdot)$ is concentrated on $\omega$ for $K\in\mathcal{K}(C,\omega)$,  we conclude the desired result
	\begin{align*}
		\lim _{t \rightarrow 0} \frac{V_{G}\left(\left[h_{t}\right]\right)-V_{G}(K)}{t}&=\frac{1}{(\sqrt{2 \pi})^{n}}\int_{\Omega_{C}}	\lim _{t \rightarrow 0} \frac{F_{t}(v)-F_{0}(v)}{t} dv\\
		&=\frac{1}{(\sqrt{2 \pi})^{n}}\int_{\Omega_{C}}e^{-\frac{\varrho_{K}^{2}}{2}} f\left(\alpha_{K}(v)\right) \varrho_{K}^{n}(v) dv\\
		&=\int_{\Omega}f(u)\bar{h}_{K}(u)dS_{\gamma^{n}}(K,  u)\\
		&=\int_{\Omega}f(u)dC_{\gamma^{n}}(K,  u)\\
		&=\int_{\omega}f(u)dC_{\gamma^{n}}(K,  u).
	\end{align*}
\end{proof}
Similar as Theorem \ref{vari formula}, the following lemma states that the  Gaussian cone measure of $C$-pseudo-cones can be generated by the variation of Gaussian volume   with respect to a logarithmic
family of Wulff shapes in \eqref{logminkowskisum}.
\begin{lemma}\label{vari formulalog}
	Let  $K \in \mathcal{K}(C, \omega)$ for some nonempty, compact set  $\omega \subset \Omega$.   Consider a continuous function  $f: \omega \rightarrow \mathbb{R}$,   and let $\left[h_{t}\right]$   be the  Wulff shapes associated with  $\left(C, \omega, h_{t}\right) $, where
	$$\log h_{t}(u)=\log \bar{h}_{K}(u)+t f(u)+o(t, u).$$ Then
	\begin{equation}
		\lim _{t \rightarrow 0} \frac{\gamma
			 ^{n}\left(\left[h_{t}\right]\right)-\gamma^{n}(K)}{t}=-\int_{\omega} f(u) dC_{\gamma^{n}}(K,  u).
	\end{equation}
	
\end{lemma}

\section{Gaussian log-Minkowski problem for $C$-pseudo-cones}

Let $\mu$ be a nonzero finite Borel measure on a nonempty compact set $\omega\subset\Omega$.
We define the functional $L_{\mu}:  C^{+}(\omega)\rightarrow (0,\infty)$ by
\begin{equation}\label{varfun def}
	L_{\mu}(f):=\gamma^{n}([f])\exp\int_{\omega}\log f d\mu.
\end{equation}
 We consider the following optimization problem
$$\sup \{L_{\mu}(f):f\in C^{+}(\omega)\}.$$
Note that for any $K\in \mathcal{K}(C,\omega)$,  it follows that $L_{\mu}(\bar{h}_{K})=\gamma^{n}(K)\exp\int_{\omega}\log \bar{h}_{K} d\mu>0$.

\begin{lemma}\label{compact cone minkowskilog}
	Let $\mu$ be a nonzero finite Borel measure on $\omega$ with some nonempty compact set $\omega \subset \Omega$. Then there exists a $C$-pseudo-cone $K\in\mathcal{K}(C,\omega)$ such that
	$$\frac{C_{\gamma^{n}}(K,\cdot)}{\gamma^{n}(K)}=\mu.$$
\end{lemma}
\begin{proof}
	
	For any $f\in C^{+}(\omega)$,  the Wulff shape satisfies $\bar{h}_{[f]}\ge f$. Consequently,
	\begin{equation}
		L_{\mu}(f)=\gamma^{n}([f])\exp\int_{\omega}\log f d\mu\le
		\gamma^{n}([f])\exp\int_{\omega}\log\bar{h}_{[f]}  d\mu=L_{\mu}(\bar{h}_{[f]}).
	\end{equation}
	This inequality implies that the supremum of $L_{\mu}(f)$ is achieved by support functions of $K(C,\omega)$.
	Let $\bar{h}_{K_{i}}$ be a maximizing sequence for $L_{\mu}$, i.e.,
	$$\lim _{i \rightarrow \infty} L_{\mu}(\bar{h}_{K_{i}})=\sup \left\{L_{\mu}(f): f \in C^{+}\left(\omega\right)\right\}>0.$$
	Define
	$$r_{i}=\min\{r:rB\cap K_{i}\neq\emptyset \}.$$
	Suppose there exists a subsequence, still denoted by $r_{i}$, such that $r_{i}\rightarrow\infty$.  By the definition of $r_{i}$,  the support function of $K_{i} $ satisfies
	$\bar{h}_{K_{i}}\le r_{i},$
	and the set inclusion
	$K_{i}\subset \mathbb{R}^{n}\backslash r_{i}B$
	gives $$\gamma^{n}(K_{i})\le \gamma^{n}(\mathbb{R}^{n}\backslash r_{i}B)=1-\gamma^{n}(r_{i}B).$$
	We obtain the following estimate for $L_{\mu}(\bar{h}_{K_{i}})$ as $r_i\to\infty$
	\begin{align*} L_{\mu}(\bar{h}_{K_{i}})&=\gamma^{n}(K_{i})\exp\int_{\omega}\log\bar{h}_{K_{i}}  d\mu\\
		&\le (1-\gamma^{n}(r_{i}B)) r_{i}^{\mu(\omega)}\\
		&=\mathbb{P}(|x|>r_{i}) r_{i}^{\mu(\omega)}\rightarrow 0.
	\end{align*}
This contradicts the fact that
$$\lim _{i \rightarrow \infty} L_{\mu}(\bar{h}_{K_{i}})=\sup \left\{L_{\mu}(f): f \in C^{+}\left(\omega\right)\right\}>0,$$
which follows $ L_{\mu}(\bar{h}_{K})>0$ for any $K\in \mathcal{K}(C,\omega)$.
	
	On the other hand, if $r_{i}\rightarrow 0$, then
	$\mathbb{P}(|x|>r_{i})\rightarrow 1,$
	and the same estimate yields
	\begin{align*}
		 L_{\mu}(\bar{h}_{K_{i}})&=\gamma^{n}(K_{i})\exp\int_{\omega}\log\bar{h}_{K_{i}}  d\mu\\
		&\le\mathbb{P}(|x|>r_{i}) r_{i}^{\mu(\omega)}\rightarrow 0,
	\end{align*} which again leads to a contradiction.
	
	The above contradictions imply that there exist two  positive constants $a$ and $b$ such that $a<\operatorname{dist}(o,\partial K_{i})<b$. By Lemma \ref{select}, there exists a  subsequence  $K_{i_{j}}\rightarrow K$
	for some $K\in\mathcal{K}(C,\omega)$ (see Lemma \ref{compace conv}). Thus, $\lim _{i \rightarrow \infty} L_{\mu}(\bar{h}_{K_{i_{j}}})=L_{\mu}(\bar{h}_{K})=\sup L_{\mu}(f)$.
	
	Consider a continuous function $f:\omega\rightarrow \mathbb{R}$. For sufficiently small $|t|$, let $\left[h_{t}\right]$   be the  Wulff shapes associated with  $\left(C, \omega, h_{t}\right) $, where
	$\log h_{t}(u)=\log \bar{h}_{K}(u)+t f(u),$ in other words, $h_{t}=\bar{h}_{K}e^{tf}$. From Lemma \ref{vari formulalog}, the variational derivative satisfies
	\begin{equation}
		0=\left.\frac{d}{d t}\right|_{t=0}L_{\mu}(h_{t})=\left( \gamma^{n}(K)\int_{\omega}fd\mu-\int_{\omega}fdC_{\gamma^{n}}K\right) \exp\int_{\omega}\log \bar{h}_{K}.
	\end{equation}
	Hence,
	\begin{equation*}
		 \gamma^{n}(K)\int_{\omega}fd\mu
=\int_{\omega}fdC_{\gamma^{n}}(K,\cdot),\quad \forall f\in C(\omega).
	\end{equation*}
	Then we conclude
	$$\mu=\frac{C_{\gamma^{n}}(K,\cdot)}{\gamma^{n}(K)}.$$
\end{proof}

To handle a general  measure  $\mu$ on $\Omega$, we employ an approximation method and the following uniform estimate for the functionals $L_{\mu_{i}}(\cdot)$ will play a crucial role.

\begin{lemma}\label{uniestilog}
	Suppose $\mu$ is a nonzero, finite Borel measure on $\Omega$. Let $\{\omega_{i}\}_{i\in\mathbb{N}}$ be a sequence of compact subsets of $\Omega$ satisfying
	$$\omega_{i}\subset \operatorname{int}\omega_{i+1}\quad \text{and}\quad \cup_{j\in\mathbb{N}} \omega_{j}=\Omega.$$  Let $\mu_{i}=\mu\llcorner\omega_{i}$ and $$L_{\mu_{i}}(\bar{h}_{K_{i}})=\sup \left\{L_{\mu_{i}}(f)=\gamma^{n}([f])\exp\int_{\omega_{i}}\log f d\mu_{i}: f \in C^{+}\left(\omega_{i}\right)\right\}.$$ Then there exists a constant $a>0$  such that $L_{\mu_{i}}(\bar{h}_{K_{i}})>a$ for any $i\in\mathbb{N}$.
	
\end{lemma}
\begin{proof}
	Let $1|_{\omega_{i}}\in C^{+}\left(\omega_{i}\right)$ be the unit constant function on $\omega_{i}$. Recall that $[1|_{\omega_{i}}]$ denotes the Wulff shape associated with $(C,\omega_{i}, 1)$. For any $i\in\mathbb{N}$,
	\begin{align*}
		[1|_{\omega_{i}}]&=C\cap \bigcap_{u \in \omega_{i}}\{x\in\mathbb{R}^{n}: \langle x, u\rangle \le -1\} \\
		&\supset C\cap \bigcap_{u \in \overline{\Omega}}\{x\in\mathbb{R}^{n}: \langle x, u\rangle \le -1\} \\
		&=\bigcap_{u \in \overline{\Omega}}\{x\in\mathbb{R}^{n}: \langle x, u\rangle \le -1\}\\
		&\triangleq L.
	\end{align*}
	There exists $z\in \operatorname{int} C$ such $\langle z, u\rangle \le -1$ for each $u\in\overline{\Omega}$. For $x\in z+C$, $x=z+c$ for some $c\in C$. Then
	$$\langle x, u\rangle=\langle z+c, u\rangle=\langle z, u\rangle+\langle c, u\rangle\le-1+0=-1$$
	for any $u \in \overline{\Omega}$. Therefore, $z+C\subset L$,
	then we have $$\gamma^{n}([1|_{\omega_{i}}])\ge\gamma^{n}(L)\ge\gamma^{n}(z+C)>0$$
	for any $i\in\mathbb{N}$. From the definition of \eqref{varfun def}, we get
	\begin{align*}
		 L_{\mu_{i}}([1|_{\omega_{i}}])&=\gamma^{n}([1|_{\omega_{i}}])\exp\int_{\omega_{i}}\log 1 d\mu_{i}\\
		&=\gamma^{n}([1|_{\omega_{i}}])\\
		&\ge \gamma^{n}(z+C).
	\end{align*}
	Then $L_{\mu_{i}}(\bar{h}_{K_{i}})=\sup \left\{L_{\mu_{i}}(f): f \in C^{+}\left(\omega_{i}\right)\right\}\ge L_{\mu_{i}}([1|_{\omega_{i}}])\ge \gamma^{n}(z+C):=a$. The desired result follows.
	
\end{proof}

With these preparations above, we turn to the proof of the following result.

\begin{theorem}\label{cone minkowskilog}
	Let $\mu$ be a nonzero, finite Borel measure on $\Omega$. Then there exists a $C$-pseudo-cone $K$ such that
	\begin{align}\label{Cone-M}
\frac{C_{\gamma^{n}}(K,\cdot)}{\gamma^{n}(K)}=\mu.
\end{align}
\end{theorem}
\begin{proof}
	Let $\{\omega_{i}\}_{i\in\mathbb{N}}$  be a sequence of compact subsets of $\Omega$ satisfying
	$$\omega_{i}\subset \operatorname{int}\omega_{i+1}\quad \text{and}\quad \cup_{j\in\mathbb{N}} \omega_{j}=\Omega.$$  There exists an index $j_{0}$ such that $\mu_{j_{0}}\neq0$. From Lemma \ref{compact cone minkowskilog}, for each  $j\ge j_{0}$, there exists a $C$-pseudo-cone $K_{j}\in\mathcal{K}(C,\omega_{j})$ satisfying
	$$\mu_{j}=\lambda_{j}C_{\gamma^{n}}(K_{j},\cdot),$$
	where the normalization constant is given by
	$$\lambda_{j}=\frac{1}{\gamma^{n}(K_{j})}.$$

	Define
	$r_{i}=\min\{r:rB\cap K_{i}\neq\emptyset \}$ as before.
	Suppose that there exists a subsequence, still denoted by $r_{i}$, such that $r_{i}\rightarrow\infty$. For $\mu_{i}\le\mu$, we obtain
	\begin{align*} L_{\mu_{i}}(\bar{h}_{K_{i}})
		&\le \gamma^{n}(K_{i})\exp\left( \mu_{i}(\omega_{i})\log r_{i}\right) \\
		&\le \gamma^{n}(K_{i})\exp\left( \mu(\Omega)\log r_{i}\right) \\
		&\le (1-\gamma^{n}(r_{i}B)) r_{i}^{\mu(\Omega)}\\
		&=\mathbb{P}(|x|>r_{i}) r_{i}^{\mu(\Omega)}\rightarrow 0.
\end{align*}
 This contradicts the  Lemma \ref{uniesti}, i.e., $L_{\mu_{i}}(\bar{h}_{K_{i}})>a>0$ for any $i\in\mathbb{N}$.
	
	On the other hand, $L_{\mu_{i}}(\bar{h}_{K_{i}})\rightarrow 0$ as $r_{i}\rightarrow 0$,  which again leads to a contradiction.
	
Then there exist two constants $a$ and $b$ such that $0<a<\operatorname{dist}(o,\partial K_{i})<b$, it follows from Lemma \ref{select} that there exists a subsequence $K_{i_{j}}\rightarrow K$
	for some $C$-pseudo-cone $K$.	
By the approximation method as in the proof of Theorem \ref{cone minkowski}, we conclude
(\ref{Cone-M}).
\end{proof}

The following Theorem asserts that one  can construct two different C-pseudo-cones with the same Gaussian cone measure.

\begin{theorem}\label{uni}
	Let $C\subset \mathbb{R}^{n}$ be a pointed, $n$-dimensional closed convex cone, and let $\omega\subset\Omega$ be a  nonempty compact  set, 	then there exist $K,L\in\mathcal{K}(C,\omega)$,  such that $C_{\gamma^{n}}(K,\cdot)=C_{\gamma^{n}}(L,\cdot)$, but $K\neq L$.
\end{theorem}
\begin{proof}
	For any $b\in\omega$, by the rotational invariance of the Gaussian measure, we may assume (after applying a suitable rotation) that $b=-e_{n}\in\Omega$. Define the hyperplane section $H(t)=C\cap H(b,-t)$. Let
	$$A=\{x\in\mathbb{R}^{n-1}:(x,1)\in H(1)\}.$$
	Consider the Gaussian surface area, denoted by $g(t)$, of $H(t)$ (up to a constant) for $t>0$ $$g(t)=\int_{H(t)}e^{-\frac{|y|^{2}}{2}}d\mathcal{H}^{n-1}(y).$$   Let $h(t)=tg(t)$, then
	\begin{align*}	 h(t)=t\int_{H(t)}e^{-\frac{|y|^{2}}{2}}d\mathcal{H}^{n-1}(y) =te^{-\frac{t^{2}}{2}}\int_{tA}e^{-\frac{|x|^{2}}{2}}d\mathcal{H}^{n-1}(x).
	\end{align*}
	The  function satisfies $h(t)>0$ and $h(t)\rightarrow 0$ as $t\rightarrow 0$ or $t\rightarrow\infty$. By the intermediate value theorem, there exist $t_{1},t_{2}>0$ with $t_{1}\neq t_{2}$ such that $h(t_{1})=h(t_{2})$.
	
	Define the sets $K=H(t_{1})+C$ and $L=H(t_{2})+C$, then $K$ and $L$  are different C-pseudo-cones in $\mathcal{K}(C,b)$, and their Gaussian cone measures satisfy $$C_{\gamma^{n}}(K,b)=\bar{h}_{K}(b)S_{\gamma^{n}}(K,b)=\frac{1}{(\sqrt{2 \pi})^{n}}t_{1}g(t_{1})$$ and $$C_{\gamma^{n}}(L,b)=\bar{h}_{L}(b)S_{\gamma^{n}}(L,b)=\frac{1}{(\sqrt{2 \pi})^{n}}t_{2}g(t_{2}).$$
	From $h(t_{1})=h(t_{2})$, we have $t_{1}g(t_{1})=t_{2}g(t_{2})$.
	Note that $C_{\gamma^{n}}(K,\cdot)$ and $ C_{\gamma^{n}}(L,\cdot)$	have support in $b$, which implies $$C_{\gamma^{n}}(K,\cdot)= C_{\gamma^{n}}(L,\cdot).$$

\end{proof}

\begin{remark}
	Let $C\subset \mathbb{R}^{n}$ be a pointed, $n$-dimensional closed convex cone,	for each $b\in\Omega$.   After applying a suitable rotation, we may  assume $b=-e_{n}$. Define the section $H(t)=C\cap H(b, -t)$. From Theorem \ref{uni}, we obtain $h(t)=\frac{t}{(\sqrt{2\pi})^{n}}\int_{H(t)}e^{-\frac{|y|^{2}}{2}}d\mathcal{H}^{n-1}(y)$ is a bounded function, that is $h(t)\le   c$ for some $c>0$. Consequently, for any Borel measure $\mu$ on $\Omega$, if
		$\mu(b)>c$,  then there does not exist a $C$-pseudo-cone $K$ such that $\mu=C_{\gamma^{n}}(K,\cdot)$.
	
\end{remark}

\end{document}